\documentclass[12pt]{article}
\usepackage{amsmath}
\usepackage{amssymb}
\usepackage{amscd}
\usepackage{amsthm}
\theoremstyle{plain}
\newtheorem{Theorem}{Theorem}
\newtheorem{Proposition}{Proposition}
\newtheorem{Lemma}{Lemma}
\newtheorem{Corollary}{Corollary}
\newtheorem{Remark}{Remark}
\newtheorem{Example}{Example}
\usepackage{indentfirst}
\usepackage{subfigure}
\usepackage{multicol}
\usepackage{mathdots}

\title{Equivalence classes of dessins d'enfants\\ with two vertices}
\author{Madoka Horie}
\date{Department of Mathematics, Tohoku University}

\begin{document}
\maketitle

\begin{quote}
{\bf Abstract.} Let $N$ be a positive integer. 
For any positive integer $L\leq N$ and any positive divisor $r$ of $N$, we enumerate the equivalence classes of dessins d'enfants with $N$ edges, $L$ faces and two vertices whose automorphism groups are cyclic of order $r$. 
Further, for any non-negative integer $h$, we enumerate the equivalence classes of dessins with $N$ edges, $h$ faces of degree $2$ with $h\leq N$, and two vertices, whose automorphism groups are cyclic of order $r$. 
Our arguments are essentially based upon a natural one-to-one correspondence of the equivalence classes of all dessins with $N$ edges to the equivalence classes of all pairs of permutations with components generating transitive subgroups of the symmetric group of degree $N$. 
\end{quote}

\section{Introduction}
In \cite{G}, Grothendieck found 
an amazing relationship between finite coverings of the projective line which are unramified outside the points 0,1, and the infinity. From his point of view, such coverings can be understood by the corresponding dessins in terms of combinatorial or topological ways (we refer \cite{Za}).  In this view, 
%the classical enumeration problems of coverings of Riemann surfaces have been studied as Hurwitz enumeration problem in many articles (cf. \cite{D-Y-Z}, \cite{M-S-S}).In particular, 
the enumeration of dessins %, which are bipartite graphs defined on Riemann surfaces, 
 have been studied in many articles
% is the most simple case of Hurwitz enumeration problem 
(cf. \cite{D-M}, \cite{M-N}). %We aim to obtain more elementary formula for numbers of equivalence classes of dessins.
In this paper, we aim to obtain several elementary formulas for numbers of equivalence classes of dessins d'enfants
%in which we treat the equivalence classes of dessins d'enfants
 with two vertices.  %(for a dessin d'enfant or simply a dessin originally defined by Grothendieck \cite{G}. 
 For dessins and its equivalence, we refer Girondo and Gonz\'alez \cite[Definition 4.\,1]
{G-G}.%).% as well as \cite{G}).

Throughout this paper, we fix positive integers $N$ and $L$ with $L \leq N$. Let $\mathbb N$ be the set of all positive integers. For each $m\in\mathbb N$, let $D_m$ denote the set of positive divisors of $m$ and $P_m$ the set of prime divisors of $m$. 
For each $n\in D_N$, we define

$$
\Lambda^{(n)}=\Lambda^{(n)} (N,L)=\left\{\lambda=(l_1, \dots, l_\nu) \middle| \
\begin{aligned}
& \nu \in \mathbb N, l_1,\dots, l_\nu \in D_{N/n}, \\
& l_1 \geq \dots \geq l_\nu,  \\
& l_1+\dots +l_\nu=L\
 \end{aligned}
\right\}.
 $$

% let $\varLambda^{(n)}$ denote the set of partitions of $L$ into parts contained in $D_{N/n}$. 
We take any $\lambda= (l_1,\dots, l_\nu)\in\varLambda^{(n)}$ and any $p\in P_{N/n}$. 
Assume that $\lambda$ is given by $L=l_1+\cdots+l_\nu$ with a positive integer $\nu$ and a decreasing sequence $l_1\geq\cdots\geq l_\nu$ of $\nu$ positive divisors of $N/n$. 
For each $\lambda = (l_1,\dots, l_\nu) \in \Lambda ^{(n)}$ and $p \in P_{N/n}$, let $i_0(n, \lambda,p)$ denote the number of positive integers $j\leq\nu$ with $\gcd(l_j,p)=1$, and $i(n,\lambda,p)$ the number of positive integers $j\leq\nu$ with $pl_j\mid N/n$, i.e., $pl_j\in D_{N/n}$. 
Obviously, $i_0(n,\lambda,p)\leq i(n,\lambda,p)$.
We then define 
$$A_{n,\lambda,p}=\left(1-\frac{1}{p}\right)^{i(n,\lambda,p)-i_0(n,\lambda,p)}\left(\left(1-\frac{1}{p}\right)^{i_0(n,\lambda,p)+1}-\left(-\frac{1}{p}\right)^{i_0(n,\lambda,p)+1}\right).$$
Note that if $i(n,\lambda,p)=0$, $A_{n,\lambda,p}=1$. 
We also define 
$$\varDelta_\lambda=\prod_{j=1}^\nu l_j,\quad B_\lambda=\frac{\nu!}{\nu_1!\nu_2!\dots\nu_L!\,\varDelta_\lambda}.$$
Here, for each positive integer $k\leq L$, $\nu_k$ denotes the number of positive integers $j\leq\nu$ with $l_j=k$. 
%Note that $B_\lambda\varDelta_{\lambda}$ is the number of $\nu$-tuples $(l_{\alpha(1)},\dots,l_{\alpha(\nu)})$ for all permutations $\alpha$ in the symmetric group $S_\nu$. 
For each positive integer $m\leq n$, we denote by $\varLambda_m^{(n)}$ the subset of $\varLambda^{(n)}$ consisting of $\lambda=(l_1,\dots,l_{\nu})$ with $\nu=m.$
%partitions of $L$ into $m$ parts in $D_{N/n}$. 
Using the elementary symmetric polynomial $\mathfrak s_m(\xi_1,\dots,\xi_n)$ of degree $m$ in $n$ variables $\xi_1$, \dots, $\xi_n$, we let 
$$f^{(n)}_m=(1-(-1)^m)\mathfrak s_m(1,2,\dots,n)=
\begin{cases}
\displaystyle{2\mathfrak s_m(1,2,\dots,n)} & \quad\text{if}\ 2\nmid m,\\
\displaystyle{0} & \quad\text{if}\ 2\mid m.
\end{cases}$$
%Note that $\varLambda_m^{(n)}=\emptyset$ if $m>L$. 
%In $S_N$, write $\sigma_0$ for the $N$-cycle sending each positive integer $j<N$ to $j+1$: 
%$$\sigma_0=(1\ 2\ \dots\ N).$$ 

For each $r \in D_N$ and $n \in D_{N/r},$ we define
$$\varPsi_{N,L,n}=\frac{N^n}{n^{n+1}(n+1)}\sum_{m=1}^{\min(n,L)}f^{(n)}_{n-m+1}\sum_{\lambda\in\varLambda_m^{(n)}}B_\lambda\prod_{p\in P_{N/n}}A_{n,\lambda,p}.$$

Let $\mathfrak D_1=\mathfrak D_1(N,L)$ denote the set of all equivalence classes of dessins with $N$ edges, $L$ faces and two vertices and, for each $r\in D_N$, let $\mathfrak D_{1,r} = \mathfrak D_{1,r}(N,L)$ denote the set of equivalence classes in $\mathfrak D_1$ whose representatives have automorphism group of order $r$. 

\begin{Theorem}
For any $r\in D_N$, 
$$|\mathfrak D_{1,r}|=\frac{r}{N}\sum_{n\in D_{N/r}}\mu(N/(nr))\varPsi_{N,L,n},$$
 %$$|\mathfrak D_{1,r}|=r\sum_{n\in D_{N/r}}\frac{\mu(N/(nr))N^{n-1}}{n^{n+1}(n+1)}\sum_{m=1}^{\min(n,L)}f^{(n)}_{n-m+1}\sum_{\lambda\in\varLambda_m^{(n)}}B_\lambda\prod_{p\in P_{N/n}}A_{n,\lambda,p},$$
where $\mu$ denotes the M\"obius function. In particular,  
%$$|\mathfrak D_1|=\sum_{r\in D_N}|\mathfrak D_{1,r}|=\sum_{u\in D_N}\frac{1}{u}\sum_{n\in D_u}\frac{\mu(u/n)N^n}{n^{n+1}(n+1)}\sum_{m=1}^{\min(n,L)}f^{(n)}_{n-m+1}\sum_{\lambda\in\varLambda_m^{(n)}}B_\lambda\prod_{p\in P_{N/n}}A_{n,\lambda,p}.$$
$$|\mathfrak D_1|=\sum_{r\in D_N}|\mathfrak D_{1,r}|=\sum_{u\in D_N}\frac{1}{u}\sum_{n\in D_u}\mu(u/n)\varPsi_{N,L,n}.$$
\end{Theorem}

\begin{Remark}
We define automorphism $f$ of dessin $(X,\phi^{-1}([0,1]))$ orientation preserving homeomorphism such that $f=f\circ \phi$.
orientation preserving homeomorphism 

A dessin with $N$ edges and two vertices has automorphism group $($ isomorphic to a cyclic group $)$ of order $N$ if and only if the dessin is regular, that is, the natural action of its automorphism group on its edges is transitive $($ for counting the equivalence classes of regular dessins with a given automorphism group, see Jones \cite[Section~3]{J} $)$.
\end{Remark}

\begin{Example}
We consider the case $N$ is an odd prime and take a divisor $r$ of $N$ as $r=N$. 
For each positive integer $t$ less than $N-1$, let $X_t$ be the compact Riemann surface associate to the curve $x(x-1)^{t}=y^{N}$ and 
$\phi_{t}$ be the Belyi function from $X$ to $\widehat{\mathbb C}:\phi_t (x,y)=x$.
Then  the dessin $(X_t,\phi_t^{-1}([0,1]))$ has $N$ edges, two vertices, one face, automorphism group isomorphic to $\langle \sigma_0\rangle$,
and a permutation representation pair $(\sigma_0, \sigma_0^t)$ $($ see Section~2 for the notation $)$.
Above result is due to the known fact that the automorphism group of a dessn is isomorphic to the centralizer of the the monodromy group of the dessin $( \cite[Theorem~4.40]{G-G} )$.
 By Theorem~1, $|\mathfrak D_{1,N}|=N-2$ $($ see also Corollary \ref{example} $)$.\
 Actually, ${N-2}$ dessins $(X_t, \phi_t^{-1}([0,1]))$ are all representatives of the above.%%%%%%%%%%%%%%%%%%%%%%%%%%%
%caluculated in Corollary \ref{example} as  %
\end{Example}

\begin{Example}Using Mathematica, we give a numerical table for each of $|\mathfrak D_{1,r}|$ and $|\mathfrak D_{1}|$ {\rm (Table 1)}.\\
\begin{table}[h]\caption{Examples of Theorem~1}
\begin{center}
\begin{tabular}{|c|c|c||c|}
\hline
$N$&$L$&$r$&$|\frak D_{1,r}|$\\ \hline

$ $&$2$&$1$&$0$\\ 
$\raisebox{1ex}[1ex][0ex]{2}$&$2$&$2$&$1$\\ \hline

$$&$1$&$1$&$0$\\
$$&$1$&$3$&$1$\\
$\raisebox{1ex}[1ex][0ex]{3}$&$3$&$1$&$0$\\
$$&$3$&$3$&$1$\\ \hline

$$&$2$&$1$&$1$\\
$$&$2$&$2$&$0$\\
$$&$2$&$4$&$1$\\
$\raisebox{1ex}[1ex][0ex]{4}$&$4$&$1$&$0$\\
$$&$4$&$2$&$0$\\
$$&$4$&$4$&$1$\\ \hline

$$&$1$&$1$&$1$\\
$5$&$1$&$5$&$3$\\
$$&$3$&$1$&$3$\\ \hline

\end{tabular}
\begin{tabular}{|c|c|c||c|}
\hline
$N$&$L$&$r$&$|\frak D_{1,r}|$\\ \hline

$5$&$3$&$5$&$0$\\ \hline

$$&$2$&$1$&$13$\\
$$&$2$&$2$&$1$\\
$$&$2$&$3$&$1$\\
$$&$2$&$6$&$	1$\\
$$&$4$&$1$&$5$\\
$$&$4$&$2$&$1$\\
$\raisebox{1ex}[1ex][0ex]{6}$&$4$&$3$&$1$\\
$$&$4$&$6$&$0$\\
$$&$6$&$1$&$0$\\
$$&$6$&$2$&$0$\\
$$&$6$&$3$&$0$\\
$$&$6$&$6$&$1$\\ \hline

$$&$1$&$1$&$25$\\
$\raisebox{1ex}[1ex][0ex]{7}$&$1$&$7$&$5$\\ \hline

\end{tabular}
\begin{tabular}{c}
$$
\end{tabular}
\begin{tabular}{|c|c||c|}
\hline
$N$&$L$&$|\frak D_{1}|$\\ \hline

$2$&$2$&$1$\\ \hline

$$&$1$&$1$\\
$\raisebox{1ex}[1ex][0ex]{3}$&$3$&$1$\\ \hline

$$&$2$&$2$\\
$\raisebox{1ex}[1ex][0ex]{4}$&$4$&$1$\\ \hline

$$&$1$&$4$\\
$5$&$3$&$3$\\
$$&$5$&$1$\\ \hline

$$&$2$&$16$\\
$6$&$4$&$7$\\
$$&$6$&$1$\\ \hline

$$&$1$&$30$\\
$$&$3$&$67$\\
$\raisebox{1ex}[1ex][0ex]{7}$&$5$&$10$\\
$$&$7$&$1$\\ \hline
\end{tabular}
\end{center}
\end{table}
%\newlength{\minitwocolumn}
%\setlength{\minitwocolumn}{0.5\textwidth}
%\addtolength{\minitwocolumn}{-0.5\columnsep}
%\begin{table}\caption{Examples of Theorem~1}\label{}
%\begin{minipage}[t]{\minitwocolumn}\begin{center}
%\begin{tabular}{llcc}
%\hline
%{\it N}&{\it L}&{\it r}&{ $|\mathfrak D_{1,r}|$}\\ \hline
%$2$&$2$&$1$&$0$\\
%$2$&$2$&$2$&$1$\\
%$3$&$1$&$1$&$0$\\
%$3$&$1$&$3$&$1$\\
%$3$&$3$&$1$&$0$\\
%$3$&$3$&$3$&$1$\\
%$4$&$2$&$1$&$1$\\
%$4$&$2$&$2$&$0$\\
%$4$&$2$&$4$&$1$\\
%$4$&$4$&$1$&$0$\\
%$4$&$4$&$2$&$0$\\
%$4$&$4$&$4$&$1$\\
%$5$&$1$&$1$&$1$\\
%$5$&$1$&$5$&$3$\\
%$5$&$3$&$1$&$3$\\
%$6$&$2$&$3$&$1$\\
%$6$&$4$&$3$&$1$\\ 
%$6$&$6$&$2$&$0$\\
%$7$&$3$&$1$&$67$\\ 
%$7$&$5$&$7$&$0$\\
%$8$&$4$&$2$&$5$\\
%$8$&$6$&$2$&$1$\\
%$9$&$3$&$1$&$2894$\\
%$10$&$2$&$2$&$25$\\
%$11$&$5$&$1$&$63355$\\
%$12$&$2$&$4$&$5$\\
%$12$&$6$&$3$&$16$\\ 
% \hline
%\end{tabular}
%\end{center}
%\end{minipage}
%\hspace{\columnsep}
%\begin{minipage}[t]{\minitwocolumn}%\caption{Examnples of $|\mathfrak D_{1}|}
%\begin{center}
%\begin{tabular}{llcc}\hline
%{\it N}&{\it L}&{ $|\mathfrak D_{1}|$}\\ \hline
%$2$&$2$&$1$\\
%$3$&$1$&$1$\\
%$3$&$3$&$1$\\
%$4$&$2$&$2$\\ 
%$4$&$4$&$1$\\
%$5$&$1$&$4$\\
%$5$&$3$&$3$\\
%$5$&$5$&$1$\\
%$6$&$2$&$16$\\
%$6$&$4$&$7$\\
%$6$&$6$&$1$\\
%$7$&$1$&$30$\\
%$7$&$3$&$67$\\
%$7$&$5$&$10$\\
%$7$&$7$&$1$\\ \hline
 %N=12,h=3 \ & \ |\mathfrak D_2|=242162.\\
%\end{tabular}
%\end{center}
%\end{minipage}
%\end{table}
%%%%%%%%%%%%%%%%%%mathematicaを用いた計算例
\end{Example}

We define the genus $g(X',D')$ of a dessin $(X',\mathcal D')$, by the genus of the Riemann surface $X'$. 
It follows that the genera of two dessins are equal if the dessins are equivalent. 
As the Riemann-Hurwitz formula shows, the number of edges of a dessin with $L$ faces and two vertices equals $N$ if and only if $2g(X',D')+L=N$ (cf.\ \cite[Proposition 4.10]{G-G}). 
%Hence Theorem~1 yields: 

\begin{Remark}
Let $N,L$ and $r$ satisfy the same condition of Theorem~1 and $g$ be a non-negative integer satisfing $2g= N-L$.
We can understand $|\mathfrak D_{1,r}|$ for genus $g$ as below:
for any $r\in D_{2g+L}$, the number of equivalence classes of dessins of genus $g$ with $L$ faces and two vertices whose automorphism groups are of order $r$. %is equal to 
%$$r\sum_{n\in D_{N/r}}\frac{\mu(N/(nr))N^{n-1}}{n^{n+1}(n+1)}\sum_{m=1}^{\min(n,L)}f^{(n)}_{n-m+1}\sum_{\lambda\in\varLambda_m^{(n)}}B_\lambda\prod_{p\in P_{N/n}}A_{n,\lambda,p}\quad\text{or}\quad0$$
%according to whether $N=2g+L$ or not. 
Obviously, in the case
$L \not\equiv N \pmod 2,$
$|\mathfrak D_{1,r}|=|\mathfrak D_{1}|=0.$
\end{Remark} 

Next, for each $(j,n)\in\mathbb N\times\mathbb N$ with $j\leq n$, we put 
\begin{equation}
\varSigma_j^{(n)}=\frac{n!}{(j-1)!}\sum_{m=0}^{n-j-1}\frac{(-1)^{m}}{m!\, (j+m)(n-j-m)}+(-1)^{n-j}\binom{n-1}{j-1}-1,
\end{equation}
so that $\varSigma_n^{(n)}=0$.% in particular. 
We also put 
\begin{equation}
\varSigma_0^{(n)}=(n-1)!-1.
\end{equation}
Note that $\varSigma^{(n)}_{n-1}=0$ and that $\varSigma^{(n)}_{n-2}=0$ if $n\geq2$. 
%Let $V$ denote the set of $N$-cycles $\tau$ in $S_N$ such that $\tau\sigma_0$ fixes just $h$ elements of $\{1,2,\dots,N\}$, and 
Let $h \in \mathbb N$ satisfy $h\leq N$. Let $D^*$ denote the set of all $n\in D_N$ such that $N/n$ divides $h$. 
For each $n\in D^*$, we define 
\begin{align*}
\varUpsilon_{N,h,n} & =\sum_{m=hn/N}^{n-1}\binom{m}{hn/N}\frac{\varphi(N/n)N^{n-m-1}}{n^{n-m-1}}\left(\frac{N}{n}-1\right)^{m-hn/N}\left(\varSigma_m^{(n)}-\varSigma_{m+1}^{(n)}\right)\\
& \ \ \ +\binom{n}{hn/N}\left(\frac{\varphi(N/n)n}{N}\left(\left(\frac{N}{n}-1\right)^{n-hn/N}-\left(-1\right)^{n-hn/N}\right)+(-1)^{n-hn/N}\right),
\end{align*}
where $\varphi$ denotes the Euler function and it is understood that $0^0=1$. 

Let $\mathfrak D_2 = \mathfrak D_2 (N,h)$ denote the set of all equivalence classes of dessins with $N$ edges, $h$ faces of degree $2$ and two vertices. 
For each $r\in D_N$, let 
$$D_r^*=D_r\cap D^*=\{n\in D_r;\ N/n\mid h\}$$ 
and let $\mathfrak D_{2,r}=\mathfrak D_{2,r}(N,h)$ denote the set of equivalence classes in $\mathfrak D_2$ whose representatives have the automorphism group of order $r$. 
The second main results of this paper is as follows.
%We can state another main result of this paper, as follows. 

\begin{Theorem}
For any $r\in D_N$, 
$$|\mathfrak D_{2,r}|=\frac{r}{N}\sum_{n\in D_{N/r}^*}\mu(N/(nr))\varUpsilon_{N,h,n}.$$
Consequently, 
$$|\mathfrak D_2|=\sum_{r\in D_N}|\mathfrak D_{2,r}|=\sum_{u\in D_N}\frac{1}{u}\sum_{n\in D_u^*}\mu(u/n)\varUpsilon_{N,h,n}.$$
\end{Theorem}

\begin{Example}
For example, the case $N=8$ and $h=4$, %%%%%%%%%%%%%%%
$$|\mathfrak D_2|=10.$$
This is the number of all inequivalent dessins with $8$ edges, $4$ faces of degree $2$ and two vertices.
Permutation representation pairs corresponding to the 10 dessins are
\begin{align*}
&(\sigma_0, (1\ 8\ 7\ 6\ 5\ 2\ 3\ 4)),\ (\sigma_0, (1\ 8\ 7\ 6\ 2\ 3\ 5\ 4)),\\
&(\sigma_0, (1\ 8\ 7\ 2\ 3\ 6\ 5\ 4)),\ (\sigma_0, (1\ 8\ 2\ 3\ 7\ 6\ 5\ 4)),\\
&(\sigma_0, (1\ 8\ 7\ 6\ 2\ 4\ 3\ 5)),\ (\sigma_0, (1\ 8\ 7\ 2\ 5\ 4\ 3\ 6)),\\
&(\sigma_0, (1\ 8\ 7\ 2\ 4\ 3\ 6\ 5)),\ (\sigma_0, (1\ 8\ 2\ 4\ 3\ 7\ 6\ 5)),\\
&(\sigma_0, (1\ 8\ 2\ 5\ 4\ 3\ 7\ 6)),\ (\sigma_0, (1\ 8\ 3\ 2\ 5\ 4\ 7\ 6)),
\end{align*} 
where $\sigma_0=(1\ 2\ 3\ 4\ 5\ 6\ 7\ 8)$.
\end{Example}

\begin{Example}Using Mathematica, we give a numerical table for each of $|\mathfrak D_{2,r}|$ and $|\mathfrak D_2|$ {\rm (Table 2)}.\\
\begin{table}[h]\caption{Examples of Theorem~2}
\begin{center}
\begin{tabular}{|c|c|c||c|}
\hline
$N$&$h$&$r$&$|\frak D_{2,r}|$\\ \hline

$$&$0$&$1$&$0$\\ 
$$&$0$&$2$&$0$\\
$$&$1$&$1$&$0$\\
$\raisebox{1ex}[1ex][0ex]{2}$&$1$&$2$&$0$\\
$$&$2$&$1$&$0$\\
$$&$2$&$2$&$1$\\ \hline

$$&$0$&$1$&$0$\\
$$&$0$&$3$&$1$\\
$$&$1$&$1$&$0$\\
$$&$1$&$3$&$0$\\
$\raisebox{1ex}[1ex][0ex]{3}$&$2$&$1$&$0$\\
$$&$2$&$3$&$0$\\ 
$$&$3$&$1$&$0$\\
$$&$3$&$3$&$1$\\ \hline

$4$&$0$&$1$&$0$\\ \hline

\end{tabular}
\begin{tabular}{|c|c|c||c|}
\hline
$N$&$h$&$r$&$|\frak D_{2,r}|$\\ \hline

$$&$0$&$2$&$0$\\
$$&$0$&$4$&$	1$\\
$$&$1$&$1$&$1$\\
$$&$1$&$2$&$0$\\
$$&$1$&$4$&$0$\\
$$&$2$&$1$&$0$\\
$$&$2$&$2$&$0$\\
$\raisebox{1ex}[1ex][0ex]{4}$&$2$&$4$&$0$\\
$$&$3$&$1$&$0$\\
$$&$3$&$2$&$0$\\
$$&$3$&$4$&$0$\\
$$&$4$&$1$&$0$\\
$$&$4$&$2$&$0$\\ 
$$&$4$&$4$&$1$\\ \hline

$5$&$0$&$1$&$1$\\ \hline

\end{tabular}
\begin{tabular}{c}
$$
\end{tabular}
\begin{tabular}{|c|c||c|}
\hline
$N$&$h$&$|\frak D_{2}|$\\ \hline

$$&$0$&$0$\\ 
$2$&$1$&$0$\\
$$&$2$&$1$\\ \hline

$$&$0$&$1$\\
$$&$1$&$0$\\ 
$\raisebox{1ex}[1ex][0ex]{3}$&$2$&$0$\\
$$&$3$&$1$\\ \hline

$$&$0$&$1$\\
$$&$1$&$1$\\
$4$&$2$&$0$\\
$$&$3$&$0$\\
$$&$4$&$1$\\ \hline

$$&$0$&$4$\\
$5$&$1$&$1$\\
$$&$2$&$2$\\ \hline
\end{tabular}
\end{center}
\end{table}%\newlength{\minitwocolumn}
\end{Example}

On the other hand, we shall deduce the following result from Theorem~1.  

\begin{Theorem}
For any $r\in D_N$, let $\mathfrak D'_{1,r}$ denote the set of equivalence classes of dessins with $N$ edges, $L$ white vertices, one black vertex and one face whose automorphism groups are of order $r$, and $\mathfrak D''_{1,r}$ denote the set of equivalence classes of dessins with $N$ edges, $L$ black vertices, one white vertex and one face whose automorphism groups are of order $r$. 
Then 
$$\left|\mathfrak D'_{1,r}\right|=\left|\mathfrak D''_{1,r}\right|=\frac{r}{N}\sum_{n\in D_{N/r}}\mu(N/(nr))\varPsi_{N,L,n}.$$ 
\end{Theorem}

There are some consequences of Theorem~3 similar to Corollary~1 such as: 

%\begin{Corollary}
%For any integer $g\geq0$, the number of equivalence classes of dessins of genus $g$ with $L$ black vertices, one white vertex and one face equals 
%$$\sum_{u\in D_N}\frac{1}{u}\sum_{n\in D_u}\frac{\mu(u/n)N^n}{n^{n+1}(n+1)}\sum_{m=1}^{\min(n,L)}f^{(n)}_{n-m+1}\sum_{\lambda\in\varLambda_m^{(n)}}B_\lambda\prod_{p\in P_{N/n}}A_{n,\lambda,p}\quad\text{or}\quad0$$
%according to whether $N=2g+L$ or not. 
%\end{Corollary}

Furthermore Theorem~2 will yield the following result. 

\begin{Theorem}
Let $r\in D_N$. 
Let $\mathfrak D'_{2,r}$ denote the set of equivalence classes of dessins with $N$ edges, $h$ white vertices of degree $1$, one black vertex and one face whose automorphism groups are of order $r\,;$ let $\mathfrak D''_{2,r}$ denote the set of equivalence classes of dessins with $N$ edges, $h$ black vertices of degree $1$, one white vertex and one face whose automorphism groups are of order $r$. 
Then 
$$\left|\mathfrak D'_{2,r}\right|=\left|\mathfrak D''_{2,r}\right|=\frac{r}{N}\sum_{n\in D_{N/r}^*}\mu(N/(nr))\varUpsilon_{N,h,n}.$$
\end{Theorem}

\begin{Remark}
As a classical enumeration problem of coverings of Riemann surfaces, Hurwitz enumeration problems has been studied $( \cite{M}, \cite{M-S-S} )$
and Mednykh gave a general formula $( \cite[Theorem 2.1]{M-S-S} )$.
We explain one of primary difference between our result and Mednykh's formula.
Mednykh fixed the ramification type $\sigma= (t^{p}_{s})$, while we will not forcus on any specific ramification type at $\infty$.
%Now we restrict the coverings to Belyi functions, put $r=3$ and $h=0$ in the formula.
%One of primely differences between our results and the formula of Mednykh is treatment of pre-image of third branch value $\infty$.
%Mednykh fixed ramification type $\sigma= (t^{p}_{s})$.
%We focused the number of pre-images of $\infty$.
\end{Remark}

The formulas of main theorems are based on the analysis of $N$-cycles belonging to centralizer $C(\sigma_0^{n})$ of $\sigma_0^{n}$ in $S_N$, where $\sigma_0 = (1\ 2\ \dots N) \in S_N$ and $n \in D_N$.

%We described our main result and some examples in section~1.
We organize this paper as follows.
In Section~2, we recall some basic facts and prepare some lemmas and propositions to prove main theorems. The proofs of Theorem~1 and Theorem~3 are given in Section~3. The proof of Theorem~2 and Theorem~4 are given in Section~4. Finally, we mention some consequence of main theorems.

\begin{center}
{\bf Acknowledgements}
\end{center}
The author wishes thank to Professor Takuya Yamauchi for his support and his comments. %
She also would like to thank the referee for reading the article carefully and fruitful suggestions.

\section{Basic facts and statement of key propositions}
When a dessin %$(X_1,\mathcal D_1)$
is equivalent to another dessin, 
%$(X_2,\mathcal D_2)$
they have same numbers of vertices, edges and faces. 
%of $(X_1,\mathcal D_1)$ equal respectively the numbers of vertices, edges and faces of $(X_2,\mathcal D_2)$. 
Furthermore, for any integer $m>0$,
equivalent dessins have the same number of faces with degree $2m$. We define the degree of a face of a dessin d'enfants as the number of edges of the dessin incident to the face (cf.\ Lando and Zvonkin \cite[Definition 1.3.8]{L-Z}).
%the number of faces of $(X_1,\mathcal D_1)$ with degree $2m$ then equals the number of faces of $(X_2,\mathcal D_2)$ with degree $2m$; here the degree of a face of a dessin is defined as the number of edges of the dessin incident to the face  

Now let $N$ be a fixed positive integer, and let $S_N$ denote as usual the symmetric group of degree $N$. 
A pair in $S_N\times S_N$ whose components generate a transitive subgroup of $S_N$ is called a permutation representation pair. 
Let $\mathcal P$ denote the set of all permutation representation pairs. 
We say that a permutation representation pair $(\sigma, \tau)$ is equivalent to a permutation representation pair $(\sigma', \tau')$, if $(\sigma, \tau)=(\rho\sigma'\rho^{-1}, \rho\tau'\rho^{-1})$ for some $\rho\in S_N$. 
An equivalence relation on $\mathcal P$ is then defined.  
We denote by $\mathfrak P$ the set of equivalence classes of pairs in $\mathcal P$. 
There exists an one-to-one correspondence between $\mathfrak{D}$ and $\mathfrak{P}$, here $\mathfrak{D}$ is the set of equivalence classes of dessins with $N$ edges.
% the equivalence classes of dessins with N edges and equivalence classes of permutation representation pairs in $S_N \times S_N$ 
(for the details, cf\ \cite[Chapter 4]{G-G}, \cite[Chapters 1, 2]{L-Z}).  

\begin{Proposition}
Let $\mathfrak c\in\mathfrak D$ and $(\sigma,\tau)\in \mathfrak{P}$ corresponding to $\mathfrak c$. 
Then, for any $m\in\mathbb N$,  
\renewcommand{\labelenumi}{{\rm (\roman{enumi})}}
\begin{enumerate}
\item the number of white vertices of degree $m$ in each dessin in $\mathfrak c$ is equal to the number of $m$-cycles in the cycle decomposition of $\sigma;$ 
\item the number of black vertices of degree $m$ in each dessin in $\mathfrak c$ is equal to the number of $m$-cycles in the cycle decomposition of $\tau;$ 
\item the number of faces of degree $2m$ in each dessin in $\mathfrak c$ is equal to the number of $m$-cycles in the cycle decomposition of $\tau\sigma$. 
\end{enumerate}
In particular, the number of white vertices, the number of black vertices and that of faces in each dessin in $\mathfrak c$ equal respectively the numbers of cycles in the cycle decompositions of $\sigma$, $\tau$ and $\tau\sigma$. 
\end{Proposition}

By an automorpism of a dessin $(X',\mathcal D')$, where $D'=\phi'^{-1}([0,1])$ for some Belyi function $\phi': X' \rightarrow \widehat{\mathbb{C}}$ , we mean an automorphism of the Riemann surface $X'$ whose restriction on $\mathcal D'$ induces an automorphism of the bicolored bipartite graph $\mathcal D'$. 
All automorphisms of $(X',\mathcal D')$ naturally form a group, which we call the automorphism group of $(X',\mathcal D')$. 
Obviously the group structure of this automorphism group is invariant whenever $(X',\mathcal D')$ is replaced by a dessin equivalent to $(X',\mathcal D')$. 
%(for basic properties of dessins, cf.\ \cite[Chapter 4]{G-G} as well as \cite[Chapter 2]{L-Z}).  
It is known that the automorphism group of any dessin with $N$ edges and two vertices is a cyclic group of order dividing $N$. 
Take any positive divisor $r$ of $N$. 
In the present paper, fixing a positive integer $L\leq N$ and a non-negative integer $h\leq L$, we shall give explicit expressions for the number of equivalence classes of dessins with $N$ edges, $L$ faces and two vertices whose automorphism groups are of order $r$, and for the number of equivalence classes of dessins with $N$ edges, $h$ faces of degree $2$ and two vertices whose automorphism groups are of order $r$. 
%The rest of this section is devoted to stating our main results along with associated ones, while we shall prove them in the succeeding three sections; some remarkable consequences of our first two theorems will be added in the last section of the paper.   %%%%%%%%%%%%%%%%%%%%%%コイツ残しとく。。。。。。。。

We prove several preliminary results in this section, with $n$ assumed to be a positive divisor of $N$: $n\in D_N$. 
Let $E$ denote the set of non-negative integers less than $N/n$, and $E^\times$ the set of integers in $E$ relatively prime to $N/n$. 
We note that $E^\times\subset\mathbb N$ or $E^\times=\{0\}$ according to whether $n<N$ or $n=N$. 
Given $m,u\in\mathbb N$ and $m$ integers $a_1$, \dots, $a_m$, we say that $a_1$, \dots, $a_m$ are distinct modulo $u$ when the residue classes of $a_1,\dots,a_m$ modulo $u$ are distinct. 

In $S_N$, write $\sigma_0$ for the $N$-cycle sending each positive integer $j<N$ to $j+1$: 
$$\sigma_0=(1\ 2\ \dots\ N).$$ 
\begin{Lemma}
Take an $N$-cycle $\tau$ in $C(\sigma_0^n)$ and a positive integer $a\leq N$. 
Then the $n$ integers $\tau^0(a)=a$, \dots, $\tau^{n-1}(a)$ are distinct modulo $n$, and $\tau^{n}(a)\equiv a+bn\pmod{N}$ with a unique $b\in E^\times$. 
\end{Lemma}

\begin{proof}
Note that an integer $j$ satisfying $\tau^j(a)=a$ is divisible by $N$. 
If $\tau^k(a)=\tau^{k'}(a)+b'n$ with integers $k$, $k'$ and $b'$, then since $\tau^k(a)=\sigma_0^{b'n}\tau^{k'}(a)=\tau^{k'}\sigma_0^{b'n}(a)$, we obtain $\tau^{(k-k')N/n}(a)=\sigma_0^{b'nN/n}(a)=a$, so that $k\equiv k'\pmod{n}$. 
It therefore follows that $\tau^0(a)$, \dots, $\tau^{n-1}(a)$ are distinct modulo $n$ and so are $\tau(a)$, \dots, $\tau^n(a)$. 
Hence we have $\tau^n(a)\equiv\tau^0(a)\pmod{n}$, i.e., $\tau^n(a)\equiv a+bn\pmod{N}$ with a unique $b\in E$. 
As the latter congruence means $\tau^n(a)=\sigma_0^{nb}(a)$, we see that 
$$\tau^{N/\gcd(b,N/n)}(a)=\tau^{n(N/n)/\gcd(b,N/n)}(a)=\sigma_0^{Nb/\gcd(b,N/n)}(a)=a.$$ 
Therefore $b$ is relatively prime to $N/n$. 
In passing, $b$ is independent of $a$. 
\end{proof}

\begin{Lemma}
Let $m\in\mathbb N$, $m\leq n$, $d\in D_{N/n}$ and $N_0=mN/(dn)$. 
Take $m$ positive integers $a_1$, \dots, $a_m$ which are not more than $N$ but distinct modulo $n$, and take an integer $b_0\in E$ satisfying $d=\gcd(b_0,N/n)$. 
Then 
\renewcommand{\labelenumi}{{\rm (\roman{enumi})}}
\begin{enumerate}
\item there exists a unique $N_0$-cycle $\rho$ in $C(\sigma_0^{dn})$ such that, for all positive integers $j<m$, $\rho(a_j)=a_{j+1}$ or equivalently $\rho^j(a_1)=a_{j+1}$ and that $\rho(a_m)\ (=\rho^m(a_1))\equiv a_1+b_0n\pmod{N}\,;$ 
\item the $N_0$-cycle $\rho$ appears in the cycle decomposition of any permutation $\tau$ in $C(\sigma_0^{dn})$ such that $\tau(a_1)=\rho(a_1)$, \dots, $\tau(a_m)=\rho(a_m)$.
\end{enumerate}
\end{Lemma}

\begin{proof}  
Let $j$ range over the non-negative integers less than $m$, and $u$ over the non-negative integers less than $N/(dn)$. 
Since the $N_0$ integers $a_{j+1}+b_0nu$ for all pairs $(j,u)$ are distinct modulo $N$, there exists a unique $N_0$-cycle $\rho$ in $S_N$ such that 
$$\rho^{j+mu}(a_1)\equiv a_{j+1}+b_0nu\pmod{N},\quad\text{i.e.,}\quad\rho^{j+mu}(a_1)=\sigma_0^{b_0nu}(a_{j+1}).$$ 
Here $u$ can be replaced by any integer, because $\rho^{j+mu'}=\rho^{j+mu}$ and $\sigma_0^{b_0nu'}=\sigma_0^{b_0nu}$ for every integer $u'$ congruent to $u$ modulo $N/(dn)$. 
Hence, in the case $j\not=m-1$, 
\begin{align*}
 \sigma_0^{b_0n}\rho\left(\rho^{j+mu}(a_1)\right) & =\sigma_0^{b_0n}\sigma_0^{b_0nu}(a_{j+2})\equiv a_{j+2}+b_0n(u+1)\pmod{N},\\
 \rho\sigma_0^{b_0n}\left(\rho^{j+mu}(a_1)\right) & =\rho\sigma_0^{b_0n}\sigma_0^{b_0nu}(a_{j+1})=\rho\rho^{j+m(u+1)}(a_1)\\
 & \equiv a_{j+2}+b_0n(u+1)\pmod{N};
\end{align*}
at the same time, 
\begin{align*}
 \sigma_0^{b_0n}\rho\left(\rho^{m-1+mu}(a_1)\right) & =\sigma_0^{b_0n}\sigma_0^{b_0n(u+1)}(a_1)\equiv a_1+b_0n(u+2)\pmod{N},\\
 \rho\sigma_0^{b_0n}\left(\rho^{m-1+mu}(a_1)\right) & =\rho\sigma_0^{b_0n}\sigma_0^{b_0nu}(a_m)=\rho\rho^{m-1+m(u+1)}(a_1)\\ & \equiv a_1+b_0n(u+2)\pmod{N}.
\end{align*}
It therefore follows that $\sigma_0^{b_0n}\rho$ coincides with $\rho\sigma_0^{b_0n}$ on $\{a_1,\rho(a_1),\dots,\rho^{N_0-1}(a_1)\}$. 
However $\rho=(a_1\ \rho(a_1)\ \dots\ \rho^{N_0-1}(a_1))$ and $\sigma_0^{b_0n}(a_1)=\rho^m(a_1)$. 
Hence 
\begin{align*}
& \sigma_0^{b_0n}\rho\sigma_0^{-b_0n}=\left(\sigma_0^{b_0n}(a_1)\ \sigma_0^{b_0n}\rho(a_1)\ \dots\ \sigma_0^{b_0n}\rho^{N_0-1}(a_1)\right)\\ & =(\sigma_0^{b_0n}\left(a_1)\ \rho\sigma_0^{b_0n}(a_1)\ \dots\ \rho^{N_0-1}\sigma_0^{b_0n}(a_1)\right)=\left(\rho^m(a_1)\ \rho^{m+1}(a_1)\ \dots\ \rho^{m+N_0-1}(a_1)\right)\\
& =\rho.
\end{align*}  
This implies that $\rho$ belongs to $C(\sigma_0^{dn})$, since $dn=\gcd(b_0n,N)$. 

Now let $\tau$ be a permutation in $C(\sigma_0^{dn})$ such that $\tau(a_{j+1})=\rho(a_{j+1})$ (for $j\in\{0,1,\dots,m-1\}$). 
Then we have 
\begin{align*}
\tau\left(\rho^{j+mu}(a_1)\right) & =\tau\sigma_0^{b_0nu}(a_{j+1})=\sigma_0^{b_0nu}\tau(a_{j+1})=\sigma_0^{b_0nu}\rho(a_{j+1})\\
& =\rho\sigma_0^{b_0nu}(a_{j+1})=\rho\left(\rho^{j+mu}(a_1)\right),
\end{align*}
which shows that $\tau$ agrees with $\rho$ on $\{a_1,\rho(a_1),\dots,\rho^{N_0-1}(a_1)\}$, that is, the $N_0$-cycle $\rho$ appears in the cycle decomposition of $\tau$. 
In particular, $\tau=\rho$ if $\tau$ is an $N_0$-cycle. 
Thus both the assertions (i) and (ii) are proved. 
\end{proof}

We take any $n$-cycle $\omega=(a_1\ a_2\ \dots\ a_n)$ in $S_n$ with $a_1=1$. 
Let 
$$\sigma_{0,n}=(1\ 2\ \dots\ n)$$ 
in $S_n$. 
Then there exist a positive integer $\nu$, $\nu$ positive integers $s_1$, \dots, $s_\nu$ and $\nu$ injections 
$$x_1:\{1,2,\dots,s_1\}\rightarrow\{1,2,\dots,n\},\ \ \dots,\ \ x_\nu:\{1,2,\dots,s_\nu\}\rightarrow\{1,2,\dots,n\}$$ 
such that the cycle decomposition of $\omega\sigma_{0,n}$ in $S_n$ is given by 
\begin{equation}
(1\ a_2\ \dots\ a_n)\sigma_{0,n}=\left(a_{x_1(1)}\ a_{x_1(2)}\ \dots\ a_{x_1(s_1)}\right)\cdots\left(a_{x_\nu(1)}\ a_{x_\nu(2)}\ \dots\ a_{x_\nu(s_\nu)}\right)
\end{equation}
and that $x_j(1)=\min\{x_j(1),x_j(2),\dots,x_j(s_j)\}$ for each positive integer $j\leq\nu$. 
It follows that not only $\nu$ but the set $\{x_1,\dots,x_\nu\}$ is unique to $\omega$. 
We may let 
$$x_1(1)=1,\quad\text{i.e.},\quad a_{x_1(1)}=1.$$ 
Obviously 
\begin{equation}
\sum_{j=1}^\nu s_j=n,\quad\bigcup_{j=1}^\nu\left\{x_j(1),x_j(2),\dots,x_j(s_j)\right\}=\{1,2,\dots,n\}.
\end{equation}
Let $u_1=0$, take any $(n-1)$-tuple $\boldsymbol u=(u_2,\dots,u_n)$ in the direct product $E^{n-1}$ of $n-1$ copies of $E$, and take any $b\in E^\times$. 
Lemma~2 then enables us to define an $N$-cycle $\rho=\rho_{\omega,\boldsymbol u,b}$ in $C(\sigma_0^n)$ by
\begin{align*} 
\rho(a_1+u_1n) & =a_2+u_2n,\ \ \dots,\ \ \rho(a_{n-1}+u_{n-1}n)=a_n+u_nn,\\
\rho(a_n+u_nn) & =a_1+u_1n+bn=1+bn.
\end{align*}  
Note that $\rho=\sigma_0^b$ in the case $n=1$. 
For any positive integer $j\leq\nu$, put 
\begin{equation*}
d_j=\gcd\left(\sum_{k=1}^{s_j}\left(u_{x_j(k+1)}-u_{x_j(k+1)-1}+\delta_j^{(k)}\right)+b_j,\ \frac{N}{n}\right),\ \ N_j=\frac{s_jN}{d_jn}.
\end{equation*} 
Here it is understood that 
$$x_j(s_j+1)=x_j(1),\quad u_0=u_n;$$
$b_j$ denotes $b$ or $0$ according as $j=1$ or $j>1$ and, for each $k\in\{1,2,\dots,s_j\}$, $\delta_j^{(k)}$ denotes $1$ or $0$ according as $a_{x_j(k)}=n$ or $a_{x_j(k)}<n$. 
Further, with Lemma~2 in mind, let $\pi_j$ denote the $N_j$-cycle in $C(\sigma_0^{d_jn})$ such that, for each positive integer $k'<s_j$,
$$\pi_j^{k'}\left(a_{x_j(1)}\right)\equiv a_{x_j(k'+1)}+\sum_{k=1}^{k'}\left(u_{x_j(k+1)}-u_{x_j(k+1)-1}+\delta_j^{(k)}\right)n\pmod{N}$$
and that 
$$\pi_j^{s_j}\left(a_{x_j(1)}\right)\equiv a_{x_j(1)}+\sum_{k=1}^{s_j}\left(u_{x_j(k+1)}-u_{x_j(k+1)-1}+\delta_j^{(k)}\right)n+b_jn\pmod{N}.$$
In the above situation, the following holds. 

\begin{Lemma}
All disjoint cycles in the cycle decomposition of $\rho\sigma_0=\rho_{\omega,\boldsymbol u,b}\sigma_0$ are 
$$\pi_1,\, \sigma_0^{n}\pi_1\sigma_0^{-n},\, \dots,\, \sigma_0^{(d_1-1)n}\pi_1\sigma_0^{-(d_1-1)n},\ \dots,\ \pi_\nu,\ \sigma_0^{n}\pi_\nu\sigma_0^{-n},\, \dots,\, \sigma_0^{(d_\nu-1)n}\pi_\nu\sigma_0^{-(d_\nu-1)n}.$$
%hence, 
%$$\rho\sigma_0=\prod_{j=1}^{\nu}\prod_{m=0}^{d_j-1}\sigma_0^{mn}\pi_j\sigma_0^{-mn}\qquad\text{in}\ S_N.$$
\end{Lemma}
\noindent
{\it Proof.}
Let $j\in\{1,2,\dots,\nu\}$. 
If $k\in\{1,2,\dots,s_j\}$, then $\sigma_{0,n}=(1\ 2\ \dots\ n)$ sends $a_{x_j(k)}$ to $a_{x_j(k)}+1-\delta_j^{(k)}n$, $(a_{x_j(1)}\ a_{x_j(2)}\ \dots\ a_{x_j(s_j)})$ sends $a_{x_j(k)}$ to $a_{x_j(k+1)}$, and so, by (3), $(a_1\ a_2\ \dots\ a_n)$ must send $a_{x_j(k)}+1-\delta_j^{(k)}n$ to $a_{x_j(k+1)}$; hence 
$$a_{x_j(k)}+1=a_{x_j(k+1)-1}+\delta_j^{(k)}n,$$
where we put $a_0=a_n$ so that 
$$a_{x_1(s_1)}+1=a_n+\delta_1^{(s_1)}n.$$ 
For convenience sake, we extend the domain of each $\sigma\in S_N$ to $\mathbb Z$ by the rule that 
$\sigma(m)=\sigma(m')$ for every $(m,m')\in\mathbb Z\times\{1,2,\dots,N\}$ with $m\equiv m'\pmod{N}$. 
Then any $a,a'\in\mathbb Z$ and any $u\in D_{N/n}$ satisfy $a+a'un\equiv\sigma_0^{a'un}(a)\pmod{N}$, which yields 
$$\tau(a+a'un)\equiv\tau(a)+a'un\pmod{N}$$ 
for any $\tau\in C(\sigma_0^{un})$, because $\tau(a+a'un)=\tau\sigma_0^{a'un}(a)=\sigma_0^{a'un}\tau(a)$. 
Obviously $\rho\sigma_0$ belongs to $C(\sigma_0^n)$ and, unless $j=1$, $\{x_j(1),x_j(2),\dots,x_j(s_j)\}$ does not contain $1$. 
Therefore, for each $k'\in\{1,2,\dots,s_j\}$, 
\begin{align*}
& \rho\sigma_0\left(a_{x_j(k')}+\sum_{k=1}^{k'-1}\left(u_{x_j(k+1)}-u_{x_j(k+1)-1}+\delta_j^{(k)}\right)n\right)\\
\equiv &\ \rho\left(a_{x_j(k')}+1\right)+\sum_{k=1}^{k'-1}\left(u_{x_j(k+1)}-u_{x_j(k+1)-1}+\delta_j^{(k)}\right)n\\
\equiv &\ \rho\left(a_{x_j(k'+1)-1}\right)+\delta_j^{(k')}n+\sum_{k=1}^{k'-1}\left(u_{x_j(k+1)}-u_{x_j(k+1)-1}+\delta_j^{(k)}\right)n\pmod{N},
\end{align*}
whence, by the definition of $\rho$, 
\begin{align*}
& \rho\sigma_0\left(a_{x_j(k')}+\sum_{k=1}^{k'-1}\left(u_{x_j(k+1)}-u_{x_j(k+1)-1}+\delta_j^{(k)}\right)n\right)\\
\equiv & 
\begin{cases}
\displaystyle{a_{x_j(k'+1)}+\sum_{k=1}^{k'}\left(u_{x_j(k+1)}-u_{x_j(k+1)-1}+\delta_j^{(k)}\right)n\pmod{N}} & \quad\text{if}\ k'<s_j,\\
\displaystyle{a_{x_j(1)}+\sum_{k=1}^{s_j}\left(u_{x_j(k+1)}-u_{x_j(k+1)-1}+\delta_j^{(k)}\right)n+b_jn\pmod{N}} & \quad\text{if}\ k'=s_j.
\end{cases}
\end{align*}
Thus 
\begin{align*}
(\rho\sigma_0)^{k'}(a_{x_j(1)}) & \equiv a_{x_j(k'+1)}+\sum_{k=1}^{k'}\left(u_{x_j(k+1)}-u_{x_j(k+1)-1}+\delta_j^{(k)}\right)n\pmod{N}\quad\text{for}\ k'<s_j,\\  
(\rho\sigma_0)^{s_j}(a_{x_j(1)}) & \equiv a_{x_j(1)}+\sum_{k=1}^{s_j}\left(u_{x_j(k+1)}-u_{x_j(k+1)-1}+\delta_j^{(k)}\right)n+b_jn\pmod{N}.
\end{align*}
The definition of $\pi_j$ therefore implies that $\pi_j$ appears in the cycle decomposition of $\rho\sigma_0$.
Since $\sigma_0^n(\rho\sigma_0)\sigma_0^{-n}=\rho\sigma_0$, it follows that the $N_j$-cycles $\sigma_0^{mn}\pi_j\sigma_0^{-mn}$, $m\in\{1,2,\dots,d_j-1\}$, 
%$$\sigma_0^n\pi_j\sigma_0^{-n},\ \sigma_0^{2n}\pi_j\sigma_0^{-2n},\ \dots,\ \sigma_0^{(d_j-1)n}\pi_j\sigma_0^{-(d_j-1)n}$$ 
also appear in the cycle decomposition of $\rho\sigma_0$. 
Furthermore, for each positive integer $k\leq N_j$, the definitions of $\pi_j$ and $d_j$ together with Lemma~2 yield 
$$\pi_j^k\left(a_{x_j(1)}\right)-a_{x_j(1)}\equiv0\pmod{d_jn}\quad\text{or}\quad\pi_j^k\left(a_{x_j(1)}\right)-a_{x_j(1)}\not\equiv0\pmod{n}$$ 
according to whether $s_j\mid k$ or $s_j\nmid k$; while, for each non-negaitive integer $m<d_j$,  
\begin{align*}
& \sigma_0^{mn}\pi_j\sigma_0^{-mn}=\left(\sigma_0^{mn}\left(a_{x_j(1)}\right)\ \sigma_0^{mn}\pi_j\left(a_{x_j(1)}\right)\ \dots\ \sigma_0^{mn}\pi_j^{N_j-1}\left(a_{x_j(1)}\right)\right),\\
& \sigma_0^{mn}\left(a_{x_j(1)}\right)-a_{x_j(1)}\equiv mn\pmod{d_jn}.
\end{align*} 
Hence the $N_j$-cycles $\sigma_0^{mn}\pi_j\sigma_0^{-mn}$, $m\in\{0,1,\dots,d_j-1\}$, 
%$$\pi_j,\ \sigma_0^n\pi_j\sigma_0^{-n},\ \dots,\ \sigma_0^{(d_j-1)n}\pi_j\sigma_0^{-(d_j-1)n}$$
are distinct.   
Given a positive integer $j'\leq\nu$ other than $j$, any of the above $N_j$-cycles differs as well from any $\sigma_0^{mn}\pi_{j'}\sigma_0^{-mn}$, $m\in\{0,1,\dots,d_{j'}-1\}$. 
%any of 
%$$\pi_{j'},\ \sigma_0^n\pi_{j'}\sigma_0^{-n},\ \dots,\ \sigma_0^{(d_{j'}-1)n}\pi_{j'}\sigma_0^{-(d_{j'}-1)n}.$$ 
Indeed, for each quartet $(k,m,k',m')$ of non-negative integers with $k<N_j$, $m<d_j$, $k'<N_{j'}$ and $m'<d_{j'}$, 
$$\sigma_0^{mn}\pi_{j}^k\left(a_{x_{j}(1)}\right)\equiv\pi_{j}^{k}\left(a_{x_{j}(1)}\right)\not\equiv\pi_{j'}^{k'}\left(a_{x_{j'}(1)}\right)\equiv\sigma_0^{m'n}\pi_{j'}^{k'}\left(a_{x_{j'}(1)}\right)\pmod{n}.$$ 
We thus find that $\pi_j,\ \sigma_0^{n}\pi_j\sigma_0^{-n},\ \dots,\ \sigma_0^{(d_j-1)n}\pi_j\sigma^{-(d_j-1)n}$ for all $j\in\{1,2,\dots,\nu\}$ are disjoint cycles in the cycle decomposition of $\rho\sigma_0$.  
The proof of the lemma is now completed since 
$$\sum_{j=1}^\nu d_jN_j=\frac{N}{n}\sum_{j=1}^\nu s_j=N.$$ 
\qed\\

%最後にMain Theorems の証明において重要なpropositions を述べてsectionを締める。的なことをいう。
We end this section with statement of two key propositions to prove main theorems.

Let $T$ be the the set of $N$-cycles $\tau$ in $S_N$ such that the cycle decomposition of $\tau\sigma_0$ contains exactly $L$ cycles, and for each $\sigma\in S_N$, let $C(\sigma)$ denote the centralizer of $\sigma$ in $S_N$. 
We write $|F|$ for the cardinality of each set $F$. 
 
\begin{Proposition}
Let $n\in D_N$ as above. 
Then 
\begin{equation*}
|T\cap C(\sigma_0^n)|=\varPsi_{N,L,n}.
%|T\cap C(\sigma_0^n)|=\frac{N^n}{n^{n+1}(n+1)}\sum_{m=1}^{\min(n,L)}f^{(n)}_{n-m+1}\sum_{\lambda\in\varLambda_m^{(n)}}B_\lambda\prod_{p\in P_{N/n}}A_{n,\lambda,p}.
\end{equation*}
\end{Proposition}

We define 
$$V=\{\tau\in S_N \mid \tau \sigma_0 \text{ fixes just } h \text{ elements. }\}.$$

\begin{Proposition}
If a positive integer $h$ satisfy $h\leq N$ and $n$ belongs to $D^*$, then 
$$|V\cap C(\sigma_0^n)|=\varUpsilon_{N,h,n}.$$
\end{Proposition} 

Proposition~2 and Proposition~3 leads us through Proposition~1 to Theorem~1 and Theorem~2, respectively. 
A proof of Proposition~2 is given in Section~3, and Proposition~3 is proved in Section~4.

\section{Proofs of Theorems 1 and 3}

Let $\lambda$ be any partition in $\varLambda^{(n)}_\nu$, still in the situation of Lemma~3. 
Assume that $\lambda$ is given by $L=l_1+\cdots+l_\nu$ with a decreasing sequence $l_1\geq\cdots\geq l_\nu$ of $\nu$ integers in $D_{N/n}$. 
For each positive integer $j\leq\nu$ and each $p\in P_{N/n}$, we denote by $l_{j,p}$ the $p$-part of $l_j$, i.e., the highest power of $p$ dividing $l_j$. 
In addition to the lemmas in the preceding section, we prove one more lemma for the proof of Proposition~1. 
To state it, we set 
$\boldsymbol d=(d_1,\dots,d_\nu)$ in $D_{N/n}^{\ \nu}$. 
%definea $\nu$-tuple $\boldsymbol d$ of positive integers by $\boldsymbol d=(d_1,\dots,d_\nu)$. 

\begin{Lemma}
Let $\alpha$ be any permutation in $S_\nu$. 
%, and $E$ the set of non-negative integers less than and relatively prime to $N/n$. 
Then the number of all 
$(\boldsymbol u,b)\in E^{n-1}\times E^\times$ 
%$(\boldsymbol u,b)\in U\times E$ 
with $\boldsymbol d=(l_{\alpha(1)},\dots,l_{\alpha(\nu)})$ is equal to 
$$\frac{N^n}{n^n\varDelta_\lambda}\prod_{p\in P_{N/n}}A_{n,\lambda,p}.$$
\end{Lemma}

\begin{proof}
We write $W$ for the set of all 
$(w_1,\dots,w_{\nu+1})\in E^{\nu+1}$ 
%$(\nu+1)$-tuples $(w_1,\dots,w_{\nu+1})$ of non-negative integers less than $N/n$ 
satisfying 
\begin{align*}
& \gcd\left(w_1,N/n\right)=l_{\alpha(1)},\ \ \dots,\ \ \gcd\left(w_\nu,N/n\right)=l_{\alpha(\nu)},\ \ \gcd\left(w_{\nu+1},N/n\right)=1,\\
& \sum_{j=1}^{\nu+1}w_j\equiv1\pmod{N/n},
\end{align*}
Let us prove 
\begin{equation}
|W|=\frac{N^\nu}{n^\nu\varDelta_\lambda}\prod_{p\in P_{N/n}}A_{n,\lambda,p}.
\end{equation}
Now, let $p$ be any prime divisor of $N/n$. 
Taking the positive integer $e(p)$ such that $p^{e(p)}$ is the $p$-part of $N/n$, we denote by $W_p$ the set of $(\nu+1)$-tuples $(w'_1,\dots,w'_{\nu+1})$ of non-negative integers less than $p^{e(p)}$ with 
$$l_{\alpha(1),p}\mid w'_1,\quad \dots,\quad l_{\alpha(\nu),p}\mid w'_\nu,\quad \sum_{j=1}^{\nu+1}w'_j\equiv1\pmod{p^{e(p)}}.$$ 
Let $J_p$ denote the set of $j\in\{1,2,\dots,\nu\}$ with $l_{\alpha(j),p}<p^{e(p)}$, $I_p$ the set of $j\in J_p$ with $l_{\alpha(j),p}=1$, and $W^{(p)}$ the set of $(w'_1,\dots,w'_{\nu+1})\in W_p$ such that $pl_{\alpha(j),p}\nmid w'_j$ for any $j\in J_p\setminus I_p$. 
Clearly 
$$|J_p|=i(n,\lambda,p)\geq|I_p|=i_0(n,\lambda,p).$$ 
When any multiple $w_j$ of $l_{\alpha(j),p}$, for each $j\in\{1,2,\dots,\nu\}$, are given, there exists a unique non-negative integer $w_{\nu+1}<p^{e(p)}$ satisfying $\sum_{j=1}^{\nu+1}w_j\equiv1\pmod{p^{e(p)}}$. 
Therefore 
$$|W_p|=\prod_{j=1}^\nu\frac{p^{e(p)}}{l_{j,p}}=\frac{p^{\nu e(p)}}{\prod_{j=1}^\nu l_{j,p}}$$ 
and consequently
\begin{equation}
\left|W^{(p)}\right|=|W_p|\left(1-\frac{1}{p}\right)^{|J_p\setminus I_p|}=\frac{p^{\nu e(p)}}{\prod_{j=1}^\nu l_{j,p}}\left(1-\frac{1}{p}\right)^{i(n,\lambda,p)-i_0(n,\lambda,p)}.
\end{equation}
We set $I'_p=I_p\cup\{\nu+1\}$.  
For each $j\in I'_p$, let $W_j^{(p)}$ denote the subset of $W^{(p)}$ consisting of $(w'_1,\dots,w'_{\nu+1})\in W^{(p)}$ with $p\mid w'_j$. 
Let $I$ range over the non-empty subsets of $I'_p$. 
Then we naturally have 
$$\left|W^{(p)}\setminus\bigcup_{j\in I'_p}W^{(p)}_j\right|=|W^{(p)}|+\sum_I(-1)^{|I|}\left|\bigcap_{j\in I} W^{(p)}_j\right|$$
and, unless $I=I'_p$, an argument similar to the one verifying (6) yields 
$$\left|\bigcap_{j\in I} W^{(p)}_j\right|=\frac{|W^{(p)}|}{p^{|I|}}$$ 
(with any element of $I'_p\setminus I$ taken instead of the subscript $\nu+1$). 
Since $\bigcap_{j\in I'_p} W^{(p)}_j=\emptyset$, it follows from the above that 
\begin{align*}
\left|W^{(p)}\setminus\bigcup_{j\in I'_p}W^{(p)}_j\right| & =\left|W^{(p)}\right|\left(\sum_{k=0}^{|I'_p|}\binom{|I'_p|}{k}\left(-\frac{1}{p}\right)^k-\left(-\frac{1}{p}\right)^{|I'_p|}\right)\\
& =\left|W^{(p)}\right|\left(\left(1-\frac{1}{p}\right)^{i_0(\lambda,p)+1}-\left(-\frac{1}{p}\right)^{i_0(\lambda,p)+1}\right).
\end{align*}
Hence, by (6), 
$$\left|W^{(p)}\setminus\bigcup_{j\in I_p}W^{(p)}_j\right|=\frac{p^{\nu e(p)}A_{n,\lambda,p}}{\prod_{j=1}^\nu l_{j,p}}.$$ 
On the other hand, $W^{(p)}\setminus\bigcup_{j\in I'_p}W^{(p)}_j$ is none other than the set of $(\nu+1)$-tuples $(w'_1,\dots,w'_{\nu+1})$ of non-negative integers less than $p^{e(p)}$ for which 
\begin{align*}
& \gcd\left(w'_1,p^{e(p)}\right)=l_{\alpha(1),p},\quad\dots,\quad\gcd\left(w'_\nu,p^{e(p)}\right)=l_{\alpha(\nu),p},\quad p\nmid w'_{\nu+1}, \\
& \sum_{j=1}^{\nu+1} w'_j\equiv1\pmod{p^{e(p)}}.
\end{align*}
This fact implies that 
$$|W|=\prod_{p\in P_{N/n}}\left|W^{(p)}\setminus\bigcup_{j\in I'_p}W^{(p)}_j\right|,$$ 
because a $(\nu+1)$-tuple 
$(w_1,\dots,w_{\nu+1})$ in $E^{\nu+1}$ 
%$(w_1,\dots,w_{\nu+1})$ of non-negative integers less than $N/n$ 
belongs to $W$ if and only if 
\begin{align*}
& \gcd\left(w^*_{1,p},p^{e(p)}\right)=l_{\alpha(1),p},\quad\dots,\quad\gcd\left(w^*_{\nu,p},p^{e(p)}\right)=l_{\alpha(\nu),p},\quad p\nmid w^*_{\nu+1,p},\\
& \sum_{j=1}^{\nu+1}w^*_{j,p}\equiv1\pmod{p^{e(p)}},
\end{align*}
with $p$ running through $P_{N/n}$, where $w^*_{1,p}$, \dots, $w^*_{\nu+1,p}$ denote respectively the minimal non-negative residues of $w_1$, \dots, $w_{\nu+1}$ modulo $p^{e(p)}$. 
Thus (5) is proved. 

Next, putting $u_1=0$ as before, let $(u_2,\dots,u_n)$ 
run through $E^{n-1}$. 
%range over the $(n-1)$-tuples in $U$. 
For each pair $(j,k)$ of positive integers with $j\leq\nu$ and $k\leq s_j$, we take 
the integer $w_j^{(k)}$ in $E$ 
%the non-negative integer $w_j^{(k)}<N/n$ 
satisfying 
$$w_j^{(k)}\equiv u_{x_j(k+1)}-u_{x_j(k+1)-1}+\delta_j^{(k)}\pmod{N/n}.$$
It follows from (4) that all $u_{x_j(k+1)}-u_{x_j(k+1)-1}$ for $(j,k)\not=(1,s_1)$ can be arranged into the sequence $u_2-u_1,\dots,u_n-u_{n-1}$, so that the correspondence 
$$(u_2,\dots,u_n)\mapsto\left(w_1^{(1)},\dots,w_1^{(s_1-1)},w_2^{(1)},\dots,w_2^{(s_2)},\dots,w_\nu^{(1)},\dots,w_\nu^{(s_\nu)}\right)$$ 
defines a permutation on 
$E^{n-1}$. 
%$U$. 
For each positive integer $j\leq\nu$, take 
the integer $\widetilde w_j$ in $E$ 
%the non-negative integer $\widetilde w_j<N/n$ 
satisfying 
$$\widetilde w_j\equiv\sum_{k=1}^{s_j}w_j^{(k)}+b_j\pmod{N/n}.$$
Since 
\begin{align*}
\sum_{j=1}^\nu\widetilde w_j-b & \equiv\sum_{j=1}^\nu\sum_{k=1}^{s_j}\left(u_{x_j(k+1)}-u_{x_j(k+1)-1}+\delta_j^{(k)}\right)+\sum_{j=1}^\nu b_j-b\\
& \equiv\sum_{j=1}^n(u_j-u_{j-1})+1\equiv1\pmod{N/n},
\end{align*}
we have $(d_1,\dots,d_\nu)=(l_{\alpha(1)},\dots,l_{\alpha(\nu)})$ if and only if $(\widetilde w_1,\dots,\widetilde w_\nu,\tilde b)$ belongs to $W$ where $\tilde b$ denotes the minimal non-negative residue of $-b$ modulo $N/n$. 
Furthermore 
$$w_2^{(s_2)}\equiv\widetilde w_2-\sum_{k=1}^{s_2-1}w_2^{(k)}\pmod{N/n},\quad\dots,\quad w_\nu^{(s_\nu)}\equiv\widetilde w_\nu-\sum_{k=1}^{s_\nu-1}w_\nu^{(k)}\pmod{N/n}.$$ 
Thus the correspondence  
$$(u_2,\dots,u_n,b)\mapsto\left(\widetilde w_1,\dots,\widetilde w_\nu,\tilde b,w_1^{(1)},\dots,w_1^{(s_1-1)},\dots,w_\nu^{(1)},\dots,w_\nu^{(s_\nu-1)}\right)$$ 
induces a bijection from the set of $(\boldsymbol u,b)$ with $\boldsymbol d=(l_{\alpha(1)},\dots,l_{\alpha(\nu)})$ to the direct product 
$W\times E^{n-\nu}$. 
%of $W$ and the set of $(n-\nu)$-tuples of non-negative integers less than $N/n$. 
Therefore, by (5), the number in question of the lemma is
$$|W|\left(\frac{N}{n}\right)^{n-\nu}=\frac{N^n}{n^n\varDelta_\lambda}\prod_{p\in P_{N/n}}A_{n,\lambda,p}.$$ 
\end{proof}

We are now ready to prove Proposition~2.\\ 

\noindent
{\it Proof of Proposition 2.} 
Denote by $\varPi_n$ the direct product of the set of $n$-cycles in $S_n$ 
and $E^{n-1}\times E^\times$. 
%, the set $U$ of $(n-1)$-tuples of non-negative integers less than $N/n$ and $E$ in Lemma~8, i.e., the set of non-negative integers less than and relatively prime to $N/n$. 
We let $\omega=(1\ a_2\ \dots\ a_n)$ range over all $n$-cycles in $S_n$, $\boldsymbol u=(u_2,\dots,u_n)$ over all $(n-1)$-tuples in $E^{n-1}$, and $b$ over all integers in $E^\times$. 
In view of Lemma~1, a map from the set of $N$-cycles in $C(\sigma_0^n)$ to $\varPi_n$ can be defined by sending each $N$-cycle $\tau$ in $C(\sigma_0^n)$ to 
$$\left((1\ a_2'\ \dots\ a_n'),\frac{\tau(1)-a_2'}{n},\dots,\frac{\tau^{n-1}(1)-a_n'}{n},\frac{\tau^n(1)-1}{n}\right),$$
where $a_j'$ denotes, for each integer $j$ with $2\leq j<n$, the minimal positive residue of $\tau^{j-1}(1)$ modulo $n$. 
Furthermore Lemma~2 for the case $(m,d)=(n,1)$ shows that this map is a bijection. 
In other words, the correspondence  
$$((1\ a_2\ \dots\ a_n),u_2,\dots,u_n,b)\mapsto\rho=\rho_{\omega,\boldsymbol u,b} $$
defines a bijection from $\varPi_n$ to the set of all $N$-cycles in $C(\sigma_0^n)$. 
Let $m$ range over the positive integers not exceeding $n$. 
%Take any positive integer $m\leq\min(n,L)$. 
Let $R_m$ denote the set of $\omega$ with $\nu=m$, i.e., the set of $(1\ a_2\ \dots\ a_n)$ for which the cycle decomposition (3) of $(1\ a_2\ \dots\ a_n)\sigma_{0,n}$ in $S_n$ contains just $m$ cycles. 
Since
$$\left|\left\{\left(l_{\alpha(1)},\dots,l_{\alpha(\nu)}\right)\,;\ \alpha\in S_\nu\right\}\right|=\frac{\nu!}{\nu_1!\nu_2!\cdots\nu_L!}=B_\lambda\varDelta_{\lambda},$$
we then see from Lemma~4 that, for each $\omega\in R_m$ and any $\lambda\in\varLambda_m^{(n)}$ given by $L=l_1+\cdots+l_m$ with a decreasing sequence $l_1\geq\cdots\geq l_m$ of $m$ integers in $D_{N/n}$, the number of $(\boldsymbol u,b)$ with $\nu=m$ and with $(d_1,\dots,d_m)=(l_{\alpha(1)},\dots,l_{\alpha(m)})$ for some $\alpha\in S_m$ is equal to 
$$\frac{N^nB_\lambda}{n^n}\prod_{p\in P_{N/n}}A_{n,\lambda,p}.$$
However, by Lemma~3, the number of cycles in the cycle decomposition of $\rho\sigma_0$ coincides with $\sum_{j=1}^\nu d_j$. 
Therefore, for each $\omega\in R_m$, the number of $\rho=\rho_{\omega,\boldsymbol u,b}$ such that the cycle decomposition of $\rho\sigma_0$  contains exactly $L$ cycles, namely the number of $\rho$ belonging to $T$, is equal to
$$\frac{N^n}{n^n}\sum_{\lambda\in\varLambda_m^{(n)}}B_\lambda\prod_{p\in P_{N/n}}A_{n,\lambda,p}.$$
Further $\varLambda_m^{(n)}=\emptyset$ if $m>L$. 
We thus obtain  
$$|T\cap C(\sigma_0^n)|=\sum_{m=1}^{\min(n,L)}\frac{|R_m|N^n}{n^n}\sum_{\lambda\in\varLambda_m^{(n)}}B_\lambda\prod_{p\in P_{N/n}}A_{n,\lambda,p}.$$
On the other hand, Theorem~1 of Zagier \cite{Z} together with \cite[Application~3]{Z} shows that $|R_m|/(n-1)!$ equals $(1+(-1)^{n-m})/(n+1)!$ times the coefficient of $\xi^m$ in the polynomial $\xi(\xi+1)\cdots(\xi+n)$ in a variable $\xi$: 
$$|R_m|=\frac{1+(-1)^{n-m}}{n(n+1)}\mathfrak s_{n-m+1}(1,2,\dots,n)=\frac{f_{n-m+1}^{(n)}}{n(n+1)}$$
(cf.\ also \cite[A.2]{L-Z}). 
Hence the proposition is proved. 
\qed\\

Let us prove Theorem~1 after simple preliminaries. 
We note that $\{\sigma_0\}\times S_N\subseteq\mathcal P$.  
%Before proving Theorem~1, we give simple preliminaries. 
Let $(\sigma,\tau)\in\iota(\mathfrak c)$, with $\mathfrak c\in\mathfrak D_1$. 
Then, by Proposition~1,  $\sigma$ and $\tau$ are $N$-cycles and the cycle decomposition of $\tau\sigma$ contains exactly $L$ cycles. 
Since $\sigma_0=\rho\sigma\rho^{-1}$ for some $\rho\in S_N$, we have not only $(\sigma_0,\rho\tau\rho^{-1})\in\iota(\mathfrak c)$ but
 $\rho\tau\rho^{-1}\sigma_0=\rho\tau\sigma\rho^{-1}$. 
%We also note that $\{\sigma_0\}\times S_N\subseteq\mathcal P$.  
Thus the following result holds. 

\begin{Lemma}
The image $\iota(\mathfrak D_1)$ consists of the equivalence classes of all $(\sigma_0,\tau)$, $\tau\in T$. 
\end{Lemma}

As any $\sigma\in C(\sigma_0)$ and any $j\in\{1,2,\dots,N\}$ satisfy 
$$\sigma(j)=\sigma\sigma_0^{j-1}(1)=\sigma_0^{j-1}\sigma(1)=\sigma_0^{j-1+\sigma(1)-1}(1)=\sigma_0^{\sigma(1)-1}(j),$$
we see that $C(\sigma_0)$ is a cyclic group generated by $\sigma_0$: $C(\sigma_0)=\langle\sigma_0\rangle$. 

\begin{Lemma}
For any $\tau\in S_N$, $\{(\sigma_0, \sigma\tau\sigma^{-1})\,;\,\sigma\in\langle\sigma_0\rangle\}$ coincides with the set of permutation representation pairs in $\{\sigma_0\}\times S_N$ equivalent to $(\sigma_0, \tau)$. 
The cardinality of this set equals the group index $[\langle\sigma_0\rangle:\langle\sigma_0\rangle\cap C(\tau)]$. 
\end{Lemma}

\begin{proof}
The relation $\sigma\sigma_0\sigma^{-1}=\sigma_0$ with $\sigma\in S_N$ means that $\sigma\in C(\sigma_0)=\langle\sigma_0\rangle$, and so the first assertion holds. 
The second assertion is easily deduced from the first. 
\end{proof}

%Using Proposition~2 together with Lemmas~5~and~6, we can prove Theorem~1. \\

\noindent
{\it Proof of Theorem 1.}
Take any $r\in D_N$. 
For each $n\in D_{N/r}$, let $T_n$ denote the set of $\tau\in T$ with $[\langle\sigma_0\rangle:\langle\sigma_0\rangle\cap C(\tau)]=n$, i.e., $|C(\sigma_0)\cap C(\tau)|=N/n$. 
Simultaneously let $\mathfrak Q_n$ denote the set of the equivalence classes of $(\sigma_0,\tau)$ for all $\tau\in T_n$. 
As is known, the automorphism group of a dessin is isomorphic to $C(\sigma)\cap C(\tau')$ if $(\sigma,\tau')$ is a permutation representation pair of the dessin (cf.\ \cite[Theorem~4.40]{G-G}). 
Hence Lemma~5 shows that 
$$\mathfrak Q_{N/r}=\iota(\mathfrak D_{1,r}).$$ 
Now, let $u$ be any positive divisor of $N/r$. 
Clearly, a permutation $\tau'$ in $S_N$ belongs to $C(\sigma_0^u)$, i.e., $C(\tau')$ contains $\langle\sigma_0^u\rangle$ if and only if $[\langle\sigma_0\rangle:\langle\sigma_0\rangle\cap C(\tau')]$ divides $u$. 
Therefore we find that $T\cap C(\sigma_0^u)$ is the disjoint union of $T_n$ for all $n\in D_u$. 
Thus 
$$|T\cap C(\sigma_0^u)|=\sum_{n\in D_u}|T_n|,$$
so that the M\"obius inversion formula in this case yields 
$$|T_u|=\sum_{n\in D_u}\mu(u/n)|T\cap C(\sigma_0^n)|.$$
By Lemma~6, $\{\sigma_0\}\times T_u$ is the set of all $(\sigma_0,\tau)\in\{\sigma_0\}\times T$ for which the number of permutation representation pairs in $\{\sigma_0\}\times S_N$ equivalent to $(\sigma_0,\tau)$ is equal to $u$; further, all permutation representation pairs in $\{\sigma_0\}\times S_N$ equivalent to a pair in $\{\sigma_0\}\times T_u$ belong to $\{\sigma_0\}\times T_u$.
Hence we see that 
$$\left|\iota^{-1}(\mathfrak Q_u)\right|=|\mathfrak Q_u|=\frac{|\{\sigma_0\}\times T_u|}{u}=\frac{|T_u|}{u}=\frac{1}{u}\sum_{n\in D_u}\mu(u/n)|T\cap C(\sigma_0^n)|$$
and in particular that 
$$|\mathfrak D_{1,r}|=\left|\iota^{-1}(\mathfrak Q_{N/r})\right|=\frac{r}{N}\sum_{n\in D_{N/r}}\mu(N/(nr))|T\cap C(\sigma_0^n)|.$$
Proposition~2 therefore completes the proof of Theorem~1. 
 \qed\\
 
Before proving Theorem~3, let $\epsilon_1$ be the permutation on $\mathcal P$ such that 
$$\epsilon_1(\sigma,\tau)=\left(\tau\sigma,\tau^{-1}\right)\quad\text{for all}\ (\sigma,\tau)\in\mathcal P,$$ 
and $\epsilon_2$ the permutation on $\mathcal P$ such that 
$$\epsilon_2(\sigma,\tau)=\left(\sigma^{-1},\tau\sigma\right)\quad\text{for all}\ (\sigma,\tau)\in\mathcal P.$$ 
Clearly $\epsilon_1$ induces a permutation on $\mathfrak P$, which we denote by $\widetilde{\epsilon_1}$. 
We also write $\widetilde{\epsilon_2}$ for the permutation on $\mathfrak P$ induced by $\epsilon_2$.\\ 
%Let us prove Theorem~3 by means of Theorem~1.\\

\noindent
{\it Proof of Theorem 3.} 
We take any $r\in D_N$. 
Let $\mathfrak Q_{N/r}$ be the same as in the proof of Theorem~1. 
Since 
$$C(\tau\sigma)\cap C(\tau^{-1})=C(\sigma)\cap C(\tau)=C(\sigma^{-1})\cap C(\tau\sigma)$$ 
for every $(\sigma,\tau)\in\mathcal P$ and since $\mathfrak Q_{N/r}=\iota(\mathfrak D_{1,r})$ as is already shown, we see from Proposition~1 that 
$$\iota^{-1}(\widetilde{\epsilon_1}(\iota(\mathfrak D_{1,r})))=\mathfrak D'_{1,r},\quad\iota^{-1}(\widetilde{\epsilon_2}(\iota(\mathfrak D_{1,r})))=\mathfrak D''_{1,r}.$$
Hence Theorem~1 yields Theorem~3. 
\qed\\

\section{Proofs of Theorems 2 and 4}

To prove Theorem~2 as well as Proposition~3, we need several preliminary results. 
Take any $n\in\mathbb N$.
We do not assume $n$ to divide $N$ at first. 
 
\begin{Lemma}
If $j\in\mathbb N$ and $j\leq n$, then 
$$\sum_{m=j}^n(-1)^m\binom{m-1}{j-1}\binom{n}{m}=(-1)^j.$$
\end{Lemma}

\begin{proof}
This can be shown by induction on $n$, with $j$ fixed. 
%We can show this by induction on $n$, fixing $j$ and using the relations 
%$$\binom{n+1}{m}=\binom{n}{m}+\binom{n}{m-1},\quad\binom{m-1}{j-1}\binom{n}{m-1}=\binom{n}{j-1}\binom{n-j+1}{m-j}$$
%for $m\in\{j,j+1,\dots,n\}$ along with the binomial theorem. 
\end{proof}

In $S_n$, we put $\sigma_{0,n}=(1\ 2\ \dots\ n)$ as before (without assuming $n\mid N$). 
When $I\subseteq\{1,2,\dots, n\}$, let $Z_I^{(n)}$ denote the set of $n$-cycles $\omega\in S_n\setminus\{\sigma_{0,n}^{-1}\}$ such that $\omega\sigma_{0,n}(a)=a$ for all $a\in I$. 
%Naturally $Z_\emptyset^{(n)}$ is the set of $n$-cycles in $S_n\setminus\{\sigma_{0,n}^{-1}\}$. 
For each non-negative integer $m\leq n$, we denote by $\mathcal F_m^{(n)}$ the family of subsets of $\{1,2,\dots,n\}$ with cardinality $m$, and by $Y_m^{(n)}$ the union of $Z_I^{(n)}$ for all $I\in\mathcal F_m^{(n)}$. 
%set 
%$$Y_m^{(n)}=\bigcup_{I\in\mathcal F_m^{(n)}}Z_I^{(n)}.$$ 

\begin{Lemma}
For any $j\in\mathbb N$ with $j\leq n$, 
$$\left|Y_j^{(n)}\right|=\sum_{m=0}^{n-j-1}(-1)^m\binom{j+m-1}{j-1}\sum_{I\in\mathcal F_{j+m}^{(n)}}\left|Z_I^{(n)}\right|.$$
\end{Lemma}

\begin{proof}
Since $Z_{\{1,2,\dots,n\}}^{(n)}=\emptyset$, we may assume $j<n$. 
%Let $Y=\bigcup_{I\in\mathcal F_j^{(n)}}Z_I^{(n)}$. 
For each $\omega\in Y_j^{(n)}$, let $J_\omega$ denote the set of positive integers $a\leq n$ with $\omega\sigma_{0,n}(a)=a$. 
Noting that $\omega\not=\sigma_{0,n}^{-1}$, we then have $j\leq|J_\omega|\leq n-1$, i.e., $0\leq|J_\omega|-j\leq n-j-1$. 
When any non-negative integer $m\leq n-j-1$ and any $\omega\in Y_j^{(n)}$ are given, the number of $I\in\mathcal F_{j+m}^{(n)}$ satisfying $Z_I^{(n)}\ni\omega$, i.e., $J_\omega\supseteq I$ is none other than $\binom{|J_\omega|}{j+m}$, which equals $0$ in the case $m>|J_\omega|-j$. 
Therefore 
\begin{align*}
& \sum_{m=0}^{n-j-1}(-1)^m\binom{j+m-1}{j-1}\sum_{I\in\mathcal F_{j+m}^{(n)}}\left|Z_I^{(n)}\right|
= \sum_{m=0}^{n-j-1}(-1)^m\binom{j+m-1}{j-1}\sum_{I\in\mathcal F_{j+m}^{(n)}}\sum_{\ \omega\in Z_I^{(n)}}1\\
& =\sum_{m=0}^{n-j-1}(-1)^m\binom{j+m-1}{j-1}\sum_{\omega\in Y_j^{(n)}}\binom{|J_\omega|}{j+m} =\sum_{\omega\in Y_j^{(n)}}\sum_{m=0}^{|J_\omega|-j}(-1)^m\binom{j+m-1}{j-1}\binom{|J_\omega|}{j+m},
\end{align*}
and here, by Lemma~7, 
$$\sum_{m=0}^{|J_\omega|-j}(-1)^m\binom{j+m-1}{j-1}\binom{|J_\omega|}{j+m}=\sum_{m'=j}^{|J_\omega|}(-1)^{m'-j}\binom{m'-1}{j-1}\binom{|J_\omega|}{m'}=1.$$
Thus the present lemma is proved. 
%$$\sum_{m=0}^{n-j-1}(-1)^m\binom{j+m-1}{j-1}\sum_{I\in\mathcal F_{j+m}^{(n)}}\left|Z_I^{(n)}\right|=\left|Y_j^{(n)}\right|.$$
\end{proof}

\begin{Lemma}
Let $I\subseteq\{1,2, \dots, n\}$ and $|I|<n$. 
Then 
$$\left|Z_I^{(n)}\right|=(n-|I|-1)!-1.$$ 
\end{Lemma}

\begin{proof}
The lemma certainly holds in the case $I=\emptyset$, since $Z_\emptyset^{(n)}$ is the set of all $n$-cycles in $S_n\setminus\{\sigma_{0,n}^{-1}\}$. 
Let us consider the case $I\not=\emptyset$ from now on. 
We put $Z_0=Z_I^{(n)}\cup\{\sigma_{0,n}^{-1}\}$ to prove $|Z_0|=(n-|I|-1)!$. 
As is easily seen, there exist a positive integer $m$ and $m$ non-empty subsets $I_1$, \dots, $I_m$ of $I$ such that $I$ is the disjoint union of $I_1$, \dots, $I_m$ and that, for each positive integer $u\leq m$, a unique $a_u\in I_u$ satisfies  
$$I_u=\left\{a_u,\sigma_{0,n}(a_u),\dots,\sigma_{0,n}^{|I_u|-1}(a_u)\right\},\ \sigma_{0,n}^{-1}(a_u)\not\in I,\ \sigma_{0,n}^{|I_u|}(a_u)\not\in I.$$
Hence, by the definition of $Z_I^{(n)}$, $Z_0$ is the set of $n$-cycles $\omega$ in $S_n$ such that, for all positive integers $u\leq m$, 
$$\omega\left(\sigma_{0,n}^{|I_u|}(a_u)\right)=\sigma_{0,n}^{|I_u|-1}(a_u),\ \omega\left(\sigma_{0,n}^{|I_u|-1}(a_u)\right)=\sigma_{0,n}^{|I_u|-2}(a_u),\ \dots,\ \omega(\sigma_{0,n}(a_u))=a_u.$$  
Since $\{\sigma_{0,n}^{-1}(a_1),\dots,\sigma_{0,n}^{-1}(a_m)\}\cap I=\emptyset$, all $\omega\in Z_0$ satisfy $\{\omega(a_1),\dots,\omega(a_m)\}\cap I=\emptyset$.   
We also note that $\sigma_{0,n}^{|I_1|}(a_1)$, \dots, $\sigma_{0,n}^{|I_m|}(a_m)$ are distinct. 

Now let $\bar I=\{1,2,\dots,n\}\setminus I$. 
We denote by $Z'$ the set of all $(n-|I|)$-cycles in the symmetric group on $\bar I$. 
If $\omega\in Z_0$, let us define a permutation $\bar{\omega}$ on $\bar I$ by the following rule: 
\begin{align*}
& \bar{\omega}\left(\sigma_{0,n}^{|I_u|}(a_u)\right)=\omega(a_u)\quad\text{for}\ u\in\{1,2,\dots,m\},\\  
& \bar{\omega}(a')=\omega(a')\qquad\qquad\ \text{for}\ a'\in\bar I\setminus\left\{\sigma_{0,n}^{|I_1|}(a_1),\dots,\sigma_{0,n}^{|I_m|}(a_m)\right\}.
\end{align*} 
We then find without difficulty that $\bar{\omega}$ belongs to $Z'$ and that the map $\omega\mapsto\bar{\omega}$ of $Z_0$ into $Z'$ is  a bijection $Z_0\rightarrow Z'$. 
Hence $|Z_0|=|Z'|=(n-|I|-1)!$.
\end{proof}

\begin{Lemma}
For any $j\in\mathbb N\cup\{0\}$ with $j<n$, the number of $n$-cycles $\omega$ in $S_n$ such that $\omega\sigma_{0,n}$ fixes exactly $j$ elements of $\{1,2,\dots,n\}$ is equal to $\varSigma_j^{(n)}-\varSigma_{j+1}^{(n)}$. 
\end{Lemma}

\begin{proof}
As $\varSigma_n^{(n)}=0$, it suffices to prove that $\varSigma_j^{(n)}$ is the number of $n$-cycles $\omega\in S_n\setminus\{\sigma_{0,n}^{-1}\}$ for which $\omega\sigma_{0.n}$ fixes at least $j$ elements of $\{1,2,\dots,n\}: |Y_j^{(n)}|=\varSigma_j^{(n)}$.
In the case $j=0$, this immediately follows from Lemma~9 and the definition (2).  
We next assume $j>0$. 
In view of Lemmas~8~and~9, we have 
\begin{align*}
& \left|Y_j^{(n)}\right|=\sum_{m=0}^{n-j-1}(-1)^m\binom{j+m-1}{j-1}\binom{n}{j+m}((n-j-m-1)!-1)\\
& =\sum_{m=0}^{n-j-1}\frac{(-1)^{m}n!}{(j-1)!\, m!\, (j+m)(n-j-m)}- \sum_{m'=j}^{n-1}(-1)^{m'-j}\binom{m'-1}{j-1}\binom{n}{m'}.  
\end{align*}
Therefore, together with the definition (1), Lemma~7 yields 
$$\left|Y_j^{(n)}\right|=\varSigma_j^{(n)}+1-(-1)^j\sum_{m'=j}^n(-1)^{m'}\binom{m'-1}{j-1}\binom{n}{m'}=\varSigma_j^{(n)}.$$ 
\end{proof}

%\noindent
%{\bf Remark 2.}
%As for $\varSigma_j^{(n)}-\varSigma_{j+1}^{(n)}$ in Lemma~10, (1) and (2) imply that 
%\begin{align*}
%\varSigma_j^{(n)}-\varSigma_{j+1}^{(n)}= &\ \frac{n!}{(j-1)!}\sum_{m=0}^{n-j-1}\frac{(-1)^{m}}{m!\, (j+m)(n-j-m)}+(-1)^{n-j}\binom{n}{j}\\
%& -\frac{n!}{j!}\sum_{m=0}^{n-j-2}\frac{(-1)^{m}}{m!\, (j+m+1)(n-j-m-1)}
%\end{align*}
%in the case $j>0$, and that
%$$\varSigma_0^{(n)}-\varSigma_1^{(n)}=n!\sum_{m=0}^{n-1}\frac{(-1)^m}{m!\,(n-m)}+(-1)^n.$$

Now let us return to the situation of Lemma~3, in which $n\mid N$. 
Fixing any $b\in E^\times$ and letting $\boldsymbol u=(u_2,\dots,u_n)$ range over all $(n-1)$-tuples in $E^{n-1}$, we define
$$U=\left\{\boldsymbol u = (u_2, \dots, u_n) \in E^{n-1} \middle| \sum_{\stackrel{1\leq j \leq \nu}{N_j=1}}d_j=h\right\}.$$
%denote by  $U$ the set of $\boldsymbol u$ such that the sum of $d_j$ for all $j\in\{1,2,\dots,\nu\}$ with $N_j=1$ is equal to $h$. 
We put $n'=n-hn/N$, so that $n'\in\mathbb N\cup\{0\}$ in the case $N/n\mid h$. 

\begin{Lemma}
Let $t$ be the number of $j\in\{1,2,\dots,\nu\}$ with $s_j=1$. 
Then 
%\begin{align*}
%& |U_0|\\[0.2cm]
% & =
$$|U|=
\begin{cases}
0 & \text{if}\ N/n\nmid h\ \text {or if}\ t<hn/N,\\[0.5cm]
\displaystyle{\binom{t}{hn/N}\frac{N^{n-t-1}}{n^{n-t-1}}\left(\frac{N}{n}-1\right)^{t-hn/N}} & \text{if}\ N/n\mid h,\ hn/N\leq t<n,\\[0.5cm]
\displaystyle{\binom{n}{hn/N}\left(\frac{n}{N}\left(\left(\frac{N}{n}-1\right)^{n'}-(-1)^{n'}\right)+\delta_b\right)} & \text{if}\ N/n\mid h,\ t=n;
\end{cases}
%\end{align*}
$$
here $\delta_b$ denotes $(-1)^{n'}$ or $0$ according to whether $b=N/n-1$ or not.  
\end{Lemma}

\begin{proof}
We denote by $I_0$ the set of $j\in\{1,2,\dots,\nu\}$ with $s_j=1$, whence $t=|I_0|$. 
For each $j\in\{1,2,\dots,\nu\}$, as follows from the definition of $N_j$, the codition $N_j=1$ is equivalent to the condition that $s_j=1$ and that $d_j=N/n$. 
Hence, by Lemma~3, 
$$h=\left|\left\{j\in I_0\,;\ d_j=N/n\right\}\right|\frac{N}{n}\leq\frac{tN}{n}$$
in the case $\boldsymbol u\in U$. 
This shows that $U=\emptyset$ if $N/n\nmid h$ or $t<hn/N$. 

Suppose next that $N/n\mid h$ and $t\geq hn/N$. 
For each $k\in\{1,2,\dots,n\}$, let $u_k^*$ denote the integer in $E$ such that $u_k^*\equiv u_k-u_{k-1}\pmod{N/n}$. 
In particular, $u_1^*\equiv-u_n\pmod{N/n}$. 
Set 
$$I^*=\left\{j\in I_0\,;\ \ u_{x_j(1)}^*\equiv-\delta_j^{(1)}-b_j\pmod{N/n}\right\},$$
and let $\mathcal F$ be the family of all subsets of $I_0$ with cardinality $hn/N$. 
Obviously 
$$|\mathcal F|=\binom{t}{hn/N}=\binom{t}{t-hn/N}.$$
Given any $j\in I_0$, we see from the definition of $d_j$ that $d_j=N/n$ if and only if $u_{x_j(1)}^*\equiv-\delta_j^{(1)}-b_j\pmod{N/n}$. 
Therefore, by Lemma~3,  the three conditions 
$$\boldsymbol u\in U,\quad I^*\in\mathcal F,\quad\left|\left\{j\in I_0\,;\ \ u_{x_j(1)}^*\not\equiv-\delta_j^{(1)}-b_j\pmod{N/n}\right\}\right|=t-\frac{hn}{N}$$
are equivalent.
On the other hand, when we assume $t<n$ with taking a positive integer $k'\leq n$ outside the set $\{x_j(1)\,;\ j\in I_0\}$ of cardinality $t$, the correspondence $\boldsymbol u\mapsto(u_1^*,\dots,u_{k'-1}^*,u_{k'+1}^*,\dots,u_n^*)$ defines a permutation on $E^{n-1}$. 
Thus, in the case $t<n$, 
$$|U|=\left(\frac{N}{n}\right)^{n-1-t}|\mathcal F|\left(\frac{N}{n}-1\right)^{t-hn/N}=\binom{t}{hn/N}\frac{N^{n-t-1}}{n^{n-t-1}}\left(\frac{N}{n}-1\right)^{t-hn/N}.$$

We now suppose that $N/n\mid h$ and $t=n$. 
Let $I$ be any set in $\mathcal F$, namely any subset of $I_0=\{1,2,\dots,n\}$ with cardinality $hn/N$. 
Let $U'$ denote the set of $\boldsymbol u\in E^{n-1}$ satisfying $I^*\supseteq I$. 
We put $\bar I=I_0\setminus I=\{1,2,\dots,n\}\setminus I$, and so $n'=|\bar I|$. 
Let us consider the case where $h<N$, i.e., $n'>0$. 
As there exists a positive integer $k$ in $\bar I$ and the correspondence $\boldsymbol u\mapsto(u_1^*,\dots,u_{k-1}^*,u_{k+1}^*,\dots,u_n^*)$ defines a permutation on $E^{n-1}$, it follows that 
$$|U'|=\left(\frac{N}{n}\right)^{n-1-hn/N}=\left(\frac{N}{n}\right)^{n'-1}.$$
For each $j\in\bar I$, let $U'_j$ denote the set of $\boldsymbol u\in U'$ with $u_{x_j(1)}^*\equiv-\delta_j^{(1)}-b_j\pmod{N/n}$. 
Clearly $U'\setminus\bigcup_{j\in\bar I}U'_j$ is the set of all $\boldsymbol u\in E^{n-1}$ satisfying $I^*=I$. 
Let $I'$ vary over the non-empty subsets of $\bar I$. 
Then 
$$\left|U'\setminus\bigcup_{j\in\bar I}U'_j\right|=|U'|+\sum_{I'}(-1)^{|I'|}\left|\bigcap_{j\in I'} U'_j\right|.$$
In the case $I'\not=\bar I$, the permutation on $E^{n-1}$ defined for any $k\in\bar I\setminus I'$ by the correspondence $\boldsymbol u\mapsto(u_1^*,\dots,u_{k-1}^*,u_{k+1}^*,\dots,u_n^*)$ causes us to have 
$$\left|\bigcap_{j\in I'}U'_j\right|=\left(\frac{N}{n}\right)^{n'-1-|I'|}.$$
Furthermore 
\begin{equation}
\sum_{j=1}^n u_{x_j(1)}^*=\sum_{j=1}^n u_j^*\equiv0\pmod{N/n},\quad\sum_{j=1}^n\left(-\delta_j^{(1)}-b_j\right)=-1-b.
\end{equation}
Therefore, when $-1-b\not\equiv0\pmod{N/n}$, i.e., $b\not=N/n-1$, we have $\bigcap_{j\in\bar I} U'_j=\emptyset$, so that 
\begin{align*}
\left|U'\setminus\bigcup_{j\in\bar I}U'_j\right| & =\left(\frac{N}{n}\right)^{n'-1}+\sum_{I'\not=\bar I}(-1)^{|I'|}\left(\frac{N}{n}\right)^{n'-1-|I'|}\\
& =\sum_{k=0}^{n'}\binom{n'}{k}(-1)^k\left(\frac{N}{n}\right)^{n'-1-k}-(-1)^{n'}\left(\frac{N}{n}\right)^{-1}\\
& =\frac{n}{N}\left(\left(\frac{N}{n}-1\right)^{n'}-(-1)^{n'}\right).
\end{align*}
When $b=N/n-1$, it follows for any $j_0\in\bar I$ that 
$$\bigcap_{j\in\bar I} U'_j=\bigcap_{j\in\bar I\setminus\{j_0\}}U'_j\quad\text{or}\quad\bigcap_{j\in\bar I} U'_j=U'_{j_0}=U'$$
according to whether $n'>1$ or $n'=1$; 
%$$\bigcap_{j\in\bar I} U'_j=
%\begin{cases}
%\displaystyle{\bigcap_{j\in\bar I\setminus\{j_0\}}U'_j} & \quad\text{if}\ n'>1,\\[0.5cm]
%\ U'_{j_0}=U' & \quad\text{if}\ n'=1;
%\end{cases}
%$$
hence 
\begin{align*}
\left|U'\setminus\bigcup_{j\in\bar I}U'_j\right| & =\sum_{k=0}^{n'}\binom{n'}{k}(-1)^k\left(\frac{N}{n}\right)^{n'-1-k}-(-1)^{n'}\left(\frac{N}{n}\right)^{-1}-(-1)^{n'-1}\\
& =\frac{n}{N}\left(\left(\frac{N}{n}-1\right)^{n'}-(-1)^{n'}\right)+(-1)^{n'}.
\end{align*}
Thus we find that the number of $\boldsymbol u\in E^{n-1}$ with $I^*=I$ is equal to 
$$\frac{n}{N}\left(\left(\frac{N}{n}-1\right)^{n'}-(-1)^{n'}\right)+\delta_b.$$
This fact yields 
$$|U|=\binom{n}{hn/N}\left(\frac{n}{N}\left(\left(\frac{N}{n}-1\right)^{n'}-(-1)^{n'}\right)+\delta_b\right).$$
Finally, in the case where $h=N$, i.e., $n'=0$, since $\mathcal F$ consists only of $\{1,2,\dots,n\}$ and (7) still holds, we easily obtain $|U|=\delta_b$, the same equality as the above. 
%=\binom{n}{hn/N}\left(\frac{n}{N}\left(\left(\frac{N}{n}-1\right)^{n'}-(-1)^{n'}\right)+\delta_b\right).$$
\end{proof}

By means of Lemmas~10~and~11, we can prove Proposition~3 as follows.\\ 

\noindent
{\it Proof of Proposition 3.}
Let $n\in D^*$. 
As in the proof of Proposition~2, let $\omega$ range over all $n$-cycles in $S_n$, $\boldsymbol u$ over all $(n-1)$-tuples in $E^{n-1}$, and $b$ over all integers in $E^\times$. 
Further, as in Lemma~11, let $t$ denote the number of $j\in\{1,2,\dots,\nu\}$ with $s_j=1$. 
We take any integer $m$ satisfying $hn/N\leq m<n$. 
By Lemma~10, the number of $\omega$ with $t=m$ is equal to $\varSigma_m^{(n)}-\varSigma_{m+1}^{(n)}$ because the condition $t=m$ means that $\omega\sigma_{0,n}$ fixes exactly $m$ elements of $\{1,2,\dots,n\}$. 
In addition, $t=n$ if and only if $\omega=\sigma_{0,n}^{\,-1}$. 
Lemma~11 therefore shows that 
\begin{align*}
|V\cap C(\sigma_0^n)|
= & \ \sum_{m=hn/N}^{n-1}\left(\varSigma_m^{(n)}-\varSigma_{m+1}^{(n)}\right)\varphi(N/n)\binom{m}{hn/N}\frac{N^{n-m-1}}{n^{n-m-1}}\left(\frac{N}{n}-1\right)^{m-hn/N}\\
& +\binom{n}{hn/N}\left(\frac{n}{N}\left(\left(\frac{N}{n}-1\right)^{n-hn/N}-(-1)^{n-hn/N}\right)+(-1)^{n-hn/N}\right)\\
& +\left(\varphi(N/n)-1\right)\binom{n}{hn/N}\left(\frac{n}{N}\left(\left(\frac{N}{n}-1\right)^{n-hn/N}-(-1)^{n-hn/N}\right)\right)\\
= & \ \varUpsilon_{N,h,n}.
\end{align*}
\qed\\

Now that Proposition~3 is proved, we proceed to prove Theorems~2~and~4.
Similarly to Lemma~5, the following result holds by Proposition~1. 

\begin{Lemma}
The image $\iota(\mathfrak D_2)$ consists of the equivalence classes of all $(\sigma_0,\tau)$, $\tau\in V$. 
\end{Lemma}

\noindent
{\it Proof of Theorem 2.}
Take any $r\in D_N$. 
For each $n\in D_{N/r}$, let $V_n$ denote the set of $\tau\in V$ satisfying $[\langle\sigma_0\rangle:C(\tau)\cap\langle\sigma_0\rangle]=n$, and $\mathfrak Q_n'$ the set of the equivalence classes of all $(\sigma_0,\tau)$, $\tau\in V_n$. 
Then, replacing $T$, $T_n$, $\mathfrak Q_n$ and Lemma~5 by $V$, $V_n$, $\mathfrak Q_n'$ and Lemma~12 respectively in the proof of Theorem~1, we can deduce that $\mathfrak Q_{N/r}'=\iota(\mathfrak D_{2,r})$ and consequently that 
$$\left|\mathfrak D_{2,r}\right|=\frac{r}{N}\sum_{n\in D_{N/r}}\mu(N/(nr))|V\cap C(\sigma_0^n)|.$$
This together with Proposition~3 gives Theorem~2, since Lemma~11 implies that $|V\cap C(\sigma_0^n)|=0$ for all $n\in D_N$ with $N/n\nmid h$. 
\qed\\

\noindent
{\it Proof of Theorem 4.} 
Let $\mathfrak Q_{N/r}'$ be as above. 
As in the proof of Theorem~3, Proposition~1 together with the fact $\mathfrak Q_{N/r}'=\iota(\mathfrak D_{2,r})$ shows that 
$$\iota^{-1}(\widetilde{\epsilon_1}(\iota(\mathfrak D_{2,r})))=\mathfrak D'_{2,r},\quad\iota^{-1}(\widetilde{\epsilon_2}(\iota(\mathfrak D_{2,r})))=\mathfrak D''_{2,r}.$$
Thus Theorem~4 follows from Theorem~2. 
\qed\\

\section{Some consequences}

\ \ \ \ We conclude this paper by deducing some results from Theorems~1~and~2 under remarkable special conditions. 

Let us begin by discussing the fact that 
\begin{equation}
|\mathfrak D_1|=|\mathfrak D_{1,N}|=1\quad\text{if}\ L=N.
\end{equation}  
We suppose $L=N$, so that, for any $n\in D_N$ and any $m\in\{1,2,\dots,n\}$, $\varLambda_m^{(n)}$ is empty or consists only of the partition of $N$ into $n$ parts equal to $N/n$ according to whether $m<n$ or $m=n$. 
We then easily see that 
\begin{align*} 
& \frac{N^{n-1}}{n^{n+1}(n+1)}\sum_{m=1}^nf^{(n)}_{n-m+1}\sum_{\lambda\in\varLambda_m^{(n)}}B_{\lambda}\prod_{p\in P_{N/n}}A_{n,\lambda,p}\\
& =\frac{2N^{n-1}}{n^{n+1}(n+1)}\mathfrak s_1(1,2,\dots,n)\left(\frac{N}{n}\right)^{-n}\prod_{p\in P_{N/n}}\left(1-\frac{1}{p}+\frac{1}{p}\right)=\frac{1}{N}.
\end{align*}
Furthermore, for any $r\in D_N$, $r\sum_{n\in D_{N/r}}\mu(N/(nr))/N$ equals $1$ or $0$ according to whether $r=N$ or $r<N$. 
Hence $|\mathfrak D_1|=|\mathfrak D_{1,N}|=1$ by Theorem~1. \\

\noindent
\begin{Remark}
Theorem~2 for $h=N$ also implies (8), whereas the equivalence class in $\mathfrak P$ consisting of $(\sigma,\sigma^{-1})$ for all $N$-cycles $\sigma$ in $S_N$ corresponds under $\iota$ to the unique equivalence class of dessins with $N$ edges, $N$ faces and two vertices.\\ 
\end{Remark}
%(cf.\ \cite[Chapter~4]{G-G}). \\

Although the Riemann-Hurwitz formula yields the fact that 
$$\mathfrak D_1=\emptyset\quad\text{unless}\ N\equiv L\pmod{2},$$ 
we deduce this from Theorem~1. 
Assume $N\not\equiv L\pmod{2}$. 
We take any $n\in D_N$ and any $\lambda\in\varLambda^{(n)}_m$ with a positive integer $m\leq\min(n,L)$. 
Assume further that $\lambda$ is given by $L=l_1+\cdots+l_m$ with a decreasing sequence $l_1\geq\cdots\geq l_m$ of $m$ integers in $D_{N/n}$. 
If $N/n$ is odd so that $l_1$, \dots, $l_m$ are odd, then $N\equiv n\pmod{2}$, $L\equiv m\pmod{2}$ and consequently $1-(-1)^{n-m+1}=0$. 
If $N/n$ is even, then $i_0(n,\lambda,2)\equiv L\equiv1\pmod{2}$ so that 
$$A_{n,\lambda,2}=2^{i_0(n,\lambda,2)-i(n,\lambda,2)}\left(2^{-i_0(n,\lambda,2)-1}-(-2)^{-i_0(n,\lambda,2)-1}\right)=0.$$ 
Thus, in any case, 
$$f^{(n)}_{n-m+1}B_\lambda\prod_{p\in P_{N/n}}A_{n,\lambda,p}=0.$$
Theorem~1 therefore yields $|\mathfrak D_1|=0$. 

Meanwhile, when $L=1$ and $N$ is odd, $\varLambda_1^{(n)}$ consists only of the trivial partition of $1$ for any $n\in D_N$. 
Further the unique $\lambda\in\varLambda_1^{(n)}$ satisfies $B_\lambda=1$ and $A_{1,\lambda,p}=(1-1/p)^2-1/p^2$ for any $p\in P_{N/n}$. 
Since $f_n^{(n)}=n!$, Theorem~1 then shows that 
\begin{equation}
|\mathfrak D_{1,r}|=2r\sum_{n\in D_{N/r}}\frac{\mu(N/(nr))N^{n-1}(n-1)!}{n^n(n+1)}\prod_{p\in P_{N/n}}\left(1-\frac{2}{p}\right)
\end{equation}
for every $r\in D_N$ and that 
$$|\mathfrak D_1|=\sum_{u\in D_N}\frac{2}{u}\sum_{n\in D_u}\frac{\mu(u/n)N^n(n-1)!}{n^n(n+1)}\prod_{p\in P_{N/n}}\left(1-\frac{2}{p}\right).$$ 
We can similarly treat the cases $L=2$, $L=3$, etc. 
For instance, when $L=2$ and $2\mid N$, Theorem~1 eventually shows that 
\begin{align*} 
|\mathfrak D_{1,r}|= & \,\frac{(3-(-1)^{N/2})r}{4}\sum_{2\,\nmid\,n}\frac{\mu(N/(nr))N^{n-1}(n-1)!}{n^n(n+1)}\prod_{p\in P_{N/n}\setminus\{2\}}\left(1-\frac{2}{p}\right)\\
&+2r\sum_{2\,\mid\,n}\frac{\mu(N/(nr))N^{n-1}(n-1)!}{n^n(n+1)}\left(\sum_{j=1}^n\frac{1}{j}\right)\prod_{p\in P_{N/n}}\left(1-\frac{3}{p}+\frac{3}{p^2}\right)
\end{align*}
for every $r\in D_N$, where $n$ runs through $D_{N/r}$;
\begin{align*} 
|\mathfrak D_1|= & \,\frac{3-(-1)^{N/2}}{8}\sum_{2\,\nmid\,u}\frac{1}{u}\sum_{n}\frac{\mu(u/n)N^n(n-1)!}{n^n(n+1)}\prod_{p\in P_{N/n}\setminus\{2\}}\left(1-\frac{2}{p}\right)\\
& +\sum_{2\,\mid\,u}\frac{2}{u}\sum_{2\,\mid\,n}\frac{\mu(u/n)N^n(n-1)!}{n^n(n+1)}\left(\sum_{j=1}^n\frac{1}{j}\right)\prod_{p\in P_{N/n}}\left(1-\frac{3}{p}+\frac{3}{p^2}\right),
\end{align*}
where $u$ runs through $D_N$ and $n$ through $D_u$. 

Now, when $N$ is a prime number so that $D_N=\{1,N\}$, essential part of Theorems~1 becomes very simple. 

\begin{Corollary}\label{example}
Assume that $N$ is a prime number congruent to $L$ modulo $2$. 
Then, in the case $L=1$, 
$$|\mathfrak D_{1,1}|=\frac{2}{N}\left(\frac{(N-1)!}{N+1}+1\right)-1,\quad|\mathfrak D_{1,N}|=N-2\,;$$ 
in the case $1<L<N$, 
$$|\mathfrak D_{1,1}|=\frac{2\mathfrak s_{N-L+1}(1,2,\dots,N)}{N^2(N+1)}=\frac{2}{N^2(N+1)}\sum_I\prod_{a\in I}a,\quad\mathfrak D_{1,N}=\emptyset,$$
where $I$ ranges over all subsets of $\{1,2,\dots, N\}$ with cardinality $N-L+1;$ in the case $L=N$, 
$$\mathfrak D_{1,1}=\emptyset,\quad|\mathfrak D_{1,N}|=1.$$ 
\end{Corollary}

\begin{proof}
The first assertion follows from (9), and the third from (8). 
If $1<L<N$, then $\varLambda_1^{(1)}=\emptyset$, $\varLambda_m^{(N)}=\emptyset$ for all positive integers $m<L$, and $\varLambda_L^{(N)}$ consists only of the partition of $L$ into $L$ parts equal to $1$. 
Therefore Theorem~1 implies the second assertion.  
\end{proof}

We finally give two consequences of Theorem~2. 

\begin{Corollary}
If $h$ is positive and relatively prime to $N$, then 
\begin{align*}
|\mathfrak D_2|= & \ |\mathfrak D_{2,1}|=\frac{1}{N}\left(\varSigma_h^{(N)}-\varSigma_{h+1}^{(N)}\right)\\
= &\ \frac{(N-1)!}{(h-1)!}\sum_{m=0}^{N-h-1}\frac{(-1)^{m}}{m!\, (h+m)(N-h-m)}+\frac{(-1)^{N-h}}{N}\binom{N}{h}\\
& -\frac{(N-1)!}{h!}\sum_{m=0}^{N-h-2}\frac{(-1)^{m}}{m!\, (h+m+1)(N-h-m-1)}.
\end{align*}
\end{Corollary}

\begin{proof}
Under the condition of the corollary, a positive divisor $n$ of $N$ satisfying $N/n\mid h$ must be $N$. 
Hence Theorem~2 together with (1) completes the proof. 
%proves the desired equalities. 
\end{proof}

\begin{Corollary}
Assume that $N$ is a prime number greater than $h$. 
Then, in the case $h=0$, 
$$|\mathfrak D_{2,1}|=\frac1{N}\left((N-1)!+2+(-1)^N\right)+\sum_{m=1}^{N-1}\frac{(-1)^m(N-1)!}{m!\,(N-m)}-1,\quad |\mathfrak D_{2,N}|=N-2\,;$$
in the case $h>0$, 
$$|\mathfrak D_{2,1}|=\frac{1}{N}\left(\varSigma_h^{(N)}-\varSigma_{h+1}^{(N)}\right),\quad\mathfrak D_{2,N}=\emptyset.$$ 
\end{Corollary}

\begin{proof}
Let $h=0$. 
Since $D_N=\{1,N\}$ and $\varSigma_0^{(1)}=\varSigma_1^{(1)}=0$, Theorem~2 yields 
$$|\mathfrak D_{2,N}|=N-2,\quad|\mathfrak D_{2,1}|=\frac{1}{N}\left(\varSigma_0^{(N)}-\varSigma_1^{(N)}-(N-2)\right).$$
These together with (1) and (2) prove the first assertion of the corollary. 
The second assertion follows immediately from Corollary~4. 
\end{proof}

\noindent
{\bf Remark 3.}
As is already noted, $\varSigma_j^{(n)}=0$ for every positive integer $n$ and every non-negative integer $j$ with $n-2\leq j\leq n$; hence Theorem~2 implies that $\mathfrak D_2=\emptyset$ if either $h=N-2$ or $h=N-1$. 
\\

%\noindent
%{\bf Note added.}
%Lemmas~1, 2, 3 and 9 (and some consequences of them such as (9)) are proved in our submitted paper \cite{H}, but the paper is not published yet and so, for the sake of completeness of the present paper, we have repeated the proofs of the lemmas in the text. 

%%%%%%%%%%%%%%%%%%%%%%

MADOKA HORIE\\
Graduate School of Science and Faculty of Science\\
Tohoku University\\
Sendai 980-8575,Japan\\
email address:horiemaaa@gmail.com\\
\end{document}